\documentclass[11pt]{amsart}
\usepackage{latexsym,amsmath,amssymb,amsfonts,amscd,graphics,appendix,amsxtra}
\usepackage[mathscr]{eucal}
\usepackage[normalem]{ulem}
\usepackage{fullpage}

\usepackage{xr-hyper}

\usepackage{hyperref}

 \numberwithin{equation}{section}
\newtheorem{theorem}{Theorem}[section]
\newtheorem{proposition}[theorem]{Proposition}
\newtheorem{lemma}[theorem]{Lemma}

\newtheorem{corollary}[theorem]{Corollary}
\newtheorem{conjecture}[theorem]{Conjecture}

\newtheorem{D}[theorem]{Definition}
\newenvironment{definition}{\begin{D} \rm }{\end{D}}
\newtheorem{R}[theorem]{Remark}
\newenvironment{remark}{\begin{R}\rm }{\end{R}}
\newtheorem{E}[theorem]{Example}
\newenvironment{example}{\begin{E}\rm }{\end{E}}
\makeatletter\let\@wraptoccontribs\wraptoccontribs\makeatother
\newcommand{\Dis}{\displaystyle}

\def\Ar{\mathbb{R}}
\def\Cee{\mathbb{C}}

\def\Id{\operatorname{Id}}
\def\Ker{\operatorname{Ker}}

\def\Hom{\operatorname{Hom}}
\def\Ext{\operatorname{Ext}}

\def\Gr{\operatorname{Gr}}

\def\im{\operatorname{Im}}

\def\Spec{\operatorname{Spec}}

\def\scrO{\mathcal{O}}

\def\codim{\operatorname{codim}}
\def\spcheck{^{\vee}}
\def\hX{\widehat{X}}

\def\cX{\mathcal{X}}
\def\cC{\mathcal{C}}
\def\cY{\mathcal{Y}}
\def\cU{\mathcal{U}}
\def\uOb{\underline\Omega^\bullet}
\def\uOp{\underline\Omega^p}
\def\uOnp{\underline\Omega^{n-p}}

  \def\bD{\mathbb{D}}
   \def\uh{\underline{h}}

\title{Higher Du Bois and higher rational singularities}

\author[R. Friedman]{Robert Friedman}
\address{Columbia University, Department of Mathematics, New York, NY 10027}
\email{rf@math.columbia.edu}
\author[R. Laza]{Radu Laza}
\address{Stony Brook University, Department of Mathematics, Stony Brook, NY 11794}
\email{radu.laza@stonybrook.edu}

\contrib[with an appendix by]{Morihiko Saito}
\begin{document}
\begin{abstract}
We prove that the higher direct images  $R^qf_*\Omega^p_{\cY/S}$ of the sheaves of relative K\"ahler differentials are locally free and compatible with arbitrary base change for flat proper families whose fibers have $k$-Du Bois  local complete intersection singularities,  for $p\leq k$ and all $q\geq 0$, generalizing a result of Du Bois (the case $k=0$). We then propose a definition of $k$-rational singularities extending the definition of rational singularities, and show that, if $X$ is a $k$-rational variety with either isolated or local complete intersection singularities, then $X$ is $k$-Du Bois. As applications, we discuss the behavior of Hodge numbers in families  and the unobstructedness of deformations of singular Calabi-Yau varieties. 

In an appendix, Morihiko Saito proves that, in the case of hypersurface singularities, the $k$-rationality definition proposed here is equivalent to a previously given numerical definition for $k$-rational singularities. As an immediate consequence, it follows that for hypersurface singularities, $k$-Du Bois singularities are $(k-1)$-rational. This statement has recently been proved for all local complete intersection singularities by Chen-Dirks-Musta\c{t}\u{a}. 
\end{abstract}

\thanks{Research of the second author is supported in part by NSF grant DMS-2101640}.
 \bibliographystyle{amsalpha}
\maketitle

\section{Introduction}\label{section1}
The Hodge numbers are constant in a smooth family of complex projective varieties over a connected base. A powerful  way of encoding this fundamental fact is Deligne's theorem  \cite{Deligne-L}:  If $f:\cY\to S$ is a smooth morphism  of complex projective varieties,  then the higher direct image sheaves $R^qf_*\Omega^p_{\cY/S}$ of the relative K\"ahler differentials are  locally free and compatible with  base change. This theorem fails for families of varieties which have singular fibers
 (and in positive characteristic). For $Y$ a singular compact complex algebraic variety, Du Bois \cite{duBois} showed that the K\"ahler-de Rham complex $\Omega_Y^\bullet$ should be replaced by the \textsl{filtered de Rham} or \textsl{Deligne-Du Bois} complex  $\uOb_Y$  whose graded pieces $\uOp_Y\in D_b^{\text{\rm{coh}}}(Y)$ play  the role of $\Omega_Y^p$ (see \cite[\S7.3]{PS}). Namely, the associated spectral sequence with $E_1^{p,q}=\mathbb H^q(Y;\uOp_Y)$ degenerates at $E_1$ and 
computes $H^{p+q}(Y,\Cee)$ together with the Hodge filtration associated to the  mixed Hodge structure on $H^*(Y)$. However, the associated Hodge numbers $\uh^{p,q}:=\dim \mathbb H^q(Y;\uOp_Y)$ do not behave well in families in general. 
 
 \smallskip
 
 The filtered de Rham complex is related to the   complex of K\"ahler differentials  via the canonical comparison K\"ahler-to-Du Bois map $\phi^p:\Omega_Y^p\to \uOp_Y$.  The maps $\phi^p$ are   isomorphisms for all $p$ only when $Y$ is smooth, at least when $Y$ is a local complete intersection  \cite[Theorem 3.39, Theorem F]{MP-loc}. It is thus natural to  consider the case when $\phi^p$ is a quasi-isomorphism in a certain range.   Steenbrink \cite[\S3]{Steenbrink} introduced the notion of \textsl{Du Bois singularities},  which play a role in the study of compactifications of moduli. By definition, $Y$ is Du Bois if $\phi^0$ is a quasi-isomorphism. Following \cite{MOPW} and \cite{JKSY-duBois}, we say that  $Y$ is \textsl{$k$-Du Bois} if $\phi^p$ is a quasi-isomorphism for $0\le p\le k$. Thus $0$-Du Bois singularities are exactly the Du Bois singularities in Steenbrink's terminology. A key property satisfied by  Du Bois singularities is the following:
 
 \begin{theorem}[{Du Bois \cite[Thm. 4.6]{duBois}, \cite[Lem. 1]{DBJ}}] \label{Thm-duBois}
Let $f:\cY\to S$ be a flat proper family of complex algebraic varieties. Assume that some fiber $Y_s$ has Du Bois singularities.  Then, possibly after replacing $S$ by a neighborhood of  $s$, for all $q\ge  0$, the sheaves $R^qf_*\scrO_\cY$ are locally free of finite type and compatible with base change. 
\end{theorem}

The theorem can be interpreted (in particular) as giving a  relation between the   mixed Hodge structure $H^*(Y_0)$ of a singular fiber $Y_0$ with the limit mixed Hodge $H^*_{\lim}$ associated to a one-parameter smoothing $\cX/\Delta$.  
In this version, the theorem (for dimension $2$ slc hypersurface singularities) was independently established by Shah \cite{Shah}, and plays a key role in the study of degenerations of $K3$ surfaces (e.g.\ \cite{Shah2}) and related objects (e.g. \cite{Laza}, \cite{KLSV}). Theorem \ref{Thm-duBois} continues to have important consequences for the study of compact moduli of varieties of general type (see e.g.\ \cite{KollarBook}, esp. \S2.5 of loc. cit.).

Here, we prove the following generalization of Theorem~\ref{Thm-duBois} for the case of a local complete intersection (lci) morphism: 

\begin{theorem}\label{mainThm}
Let $f:\cY\to S$ be a flat proper family of complex algebraic varieties and let $s\in S$. Suppose that the fiber $Y_s$ has $k$-Du Bois  lci  singularities. Then, possibly after replacing $S$ by a neighborhood of $s$,  the higher direct image sheaves $R^qf_*\Omega^p_{\cY/S}$ of the relative K\"ahler differentials are locally free and compatible with base change for $0\le p\le k$ and all $q\ge 0$. 
\end{theorem}

\begin{remark}
We use the lci assumption to control the sheaves of K\"ahler and relative K\"ahler differentials in two ways.  First, a result of \cite{MP-loc} gives an estimate on the   codimension of the singular locus for $k$-Du Bois lci singularities. Using this and some results of Greuel, we  prove a key technical point: under the lci assumption, the sheaves $\Omega^p_{\cX/S}$ are flat over $S$ for $p\le k$  (Theorem \ref{flatness}).  In case $k=0$, 
both the codimension estimate and the flatness are automatic  and one recovers Theorem~\ref{Thm-duBois} as a special case. 
\end{remark}

As a corollary, we obtain: 

\begin{corollary}\label{mainThmcor}
Let $f\colon \cY\to S$ be a flat proper family of complex algebraic varieties over an irreducible  base. For $s\in S$, suppose that the fiber $Y_s$ has $k$-Du Bois  lci singularities. Then,  for every fiber $t$ such that $Y_t$ is smooth, $\dim \Gr_F^p H^{p+q}(Y_t) = \dim \Gr_F^p H^{p+q}(Y_s)$ for every $q$ and for $0\le p \le k$. Equivalently $\uh^{p,q}(Y_s)=h^{p,q}(Y_t)$ for all $p\le k$. \qed
\end{corollary}

In the case of hypersurface singularities, results similar to Corollary~\ref{mainThmcor} were obtained by Kerr-Laza   \cite{KL2, KL3} with further clarifications given by Saito (personal communications) based on \cite{SaitoV}.  These type of results are for example relevant to the study of the moduli of cubic fourfolds \cite{Laza}.

\medskip

Another application of Theorem \ref{mainThm} is the following  generalization of the results of Kawamata \cite{Kawamata}, Ran \cite{Ran}, and Tian \cite{Tian} on the unobstructedness of deformations for nodal Calabi-Yau varieties  in any dimension, where the special case of isolated hypersurface singularities  was established in \S6 of the first version of \cite{FL}:

\begin{corollary}\label{CYdef} 
Let $Y$ be a canonical  Calabi-Yau variety (Definition~\ref{defcanonCY}) which is additionally a scheme with $1$-Du Bois lci (not necessarily isolated) singularities.  Then the functor $\mathbf{Def}(Y)$ is unobstructed. 
\end{corollary}

A more well-known class of singularities are the {\it rational singularities}. 
By work of Steenbrink \cite{Steenbrink}, Kov\'acs \cite{Kovacs99}, and others, a rational singularity is  Du Bois.  In the context of higher Du Bois singularities, it is natural to consider {\it higher rational singularities}. In the case of hypersurfaces $X$ in a smooth variety, the $k$-Du Bois singularities  are characterized numerically by the condition $\tilde \alpha_X\ge k+1$ (\cite[Thm. 1.1]{MOPW}, \cite[Thm. 1]{JKSY-duBois}), where  $\tilde \alpha_X$ is the {\it minimal exponent} invariant, a generalization of the log canonical threshold. Given that $\tilde \alpha_X>1$ characterizes rational hypersurface singularities (\cite{Saito-b}), it is natural to define {\it $k$-rational singularities} numerically by the strict inequality $\tilde \alpha_X> k+1$ (\cite[Def. 4.3]{KL2}); this definition also occurs implicitly in \cite{MOPW} and \cite{JKSY-duBois}. While results can be obtained using this numerical definition, it has the disadvantage of being restricted to hypersurfaces, and to be somewhat ad hoc. A more general definition of $k$-rational isolated singularities, based on local vanishing properties (see esp.\ \cite{MP-Inv}), was given in \S3 of the first version of 
\cite{FL} (see also \cite{MP-rat} for further discussion). 
  
In this paper, we propose a more intrinsic definition of $k$-rational singularities in general (Definition~\ref{defkrat}) and show that it agrees with the usual definition of rational singularities for $k=0$ and with the definition of \cite[\S3]{FL} (under mild assumptions; see Corollaries~\ref{cor-rat0} and \ref{cor-olddef}). Additionally, for hypersurface singularities, M. Saito proves that the new definition proposed here is indeed equivalent to the previous numerical definition mentioned above (Theorem~\ref{Thm-A}). The main advantage of the definition of higher rational singularities given here is that it naturally factors through the higher Du Bois condition.  In analogy with the case $k=0$, we conjecture that {\it $k$-rational implies $k$-Du Bois} in general. In the first version of \cite[\S3]{FL} (and expanded in \cite{FL22d}), we verified this conjecture under the assumption of isolated lci  singularities. Here, we generalize this in both directions, for arbitrary isolated or lci singularities:

\begin{theorem}\label{thm-krat} Let $X$ have either  lci singularities (not necessarily isolated) or isolated singularities (not necessarily lci). If $X$ is $k$-rational, then $X$ is $k$-Du Bois.
\end{theorem} 
\begin{remark}
Musta\c{t}\u{a} and Popa  gave an independent proof of Theorem~\ref{thm-krat} for the case of lci singularities (\cite[Thm. B]{MP-rat}).
\end{remark}

The isolated complete intersection case (see \cite{FL22d}) sheds light on the tight relationship between higher rational and higher Du Bois singularities. In the first version of this paper, we made the following conjecture: 

\begin{conjecture}\label{conjDBrat} If $X$ has  lci singularities and $X$ is $k$-Du Bois, then $X$ is $(k-1)$-rational.
\end{conjecture}

 For an  isolated hypersurface  singularity, Conjecture~\ref{conjDBrat}  is an immediate consequence of the following result: 

\begin{proposition}\label{hypersurfcase} Let $(X,x)$ be an isolated hypersurface singularity and let $\widetilde\alpha_{X,x} = \widetilde\alpha_X$ be the minimal  exponent   as defined by Saito \cite{Saito-b}. Then 
\begin{itemize}
\item[\rm(i)] $X$ is $k$-Du Bois   $\iff$ $\widetilde\alpha_X \geq k+1$.
\item[\rm(ii)] $X$ is $k$-rational  $\iff$ $\widetilde\alpha_X > k+1$. \qed
\end{itemize}  
\end{proposition}

Here  (i) follows from  \cite[Thm. 1]{JKSY-duBois}  and \cite[Thm. 1.1]{MOPW} and holds true for a general, not necessarily isolated, hypersurface singularity, and (ii) is proved in \cite[Corollary 6.6]{FL22d}.  In Appendix \ref{AppendixA}, M. Saito proves (ii)  for the case of a general hypersurface singularity, based on the results of \cite{JKSY-duBois}. Thus implies Conjecture~\ref{conjDBrat}  for the case of general hypersurface singularities, not necessarily isolated (Cor. \ref{Cor-A}).
 Musta\c{t}\u{a} and Popa  have proved (ii)  of Proposition~\ref{hypersurfcase}, and hence Conjecture~\ref{conjDBrat} in this case as well (cf. \cite[Thm. E, Cor. F]{MP-rat}). Additionally, in  \cite[Corollary 5.5]{FL22d}, we established Conjecture~\ref{conjDBrat} for isolated lci singularities. Recently,  Chen-Dirks-Musta\c{t}\u{a} \cite{ChenDirksM} have proved Conjecture~\ref{conjDBrat} in general.

 Since $k$-rational singularities are milder than $k$-Du Bois singularities, one expects that more of the Hodge diamond is 
preserved in families with $k$-rational singularities. Indeed, this is the case as shown in \cite[Cor. 4.2]{KL2} (isolated hypersurface $k$-rational singularities) and \cite[Thm. 1]{KLS} (arbitrary rational singularities). Here, we extend these results to $k$-rational lci singularities, and clarify the difference between $k$-Du Bois and $k$-rational in this context. Essentially, for $k$-rational singularities, in addition to the preservation given by Corollary \ref{mainThmcor}, one gains Hodge symmetry in a certain range.  Informally, we can say that the frontier Hodge diamond up to coniveau $k$ is preserved for deformations of $k$-rational singularities. 

\begin{corollary}\label{Corkrat}
Let $f\colon \cY\to S$ be a flat proper family of complex algebraic varieties over an irreducible  base. For $s\in S$, suppose that the fiber $Y_s$ has $k$-rational  lci singularities. Then,  for every fiber $t$ such that $Y_t$ is smooth, and for all $p\le k$,
$$\uh^{p,q}(Y_s)=\uh^{q,p}(Y_s)=\uh^{n-p,n-q}(Y_s) =h^{p,q}(Y_t)= h^{q,p}(Y_t)=h^{n-p,n-q}(Y_t).$$
Moreover, for all $p\le k$,
$$\uh^{p,q}(Y_s) = \dim \Gr_F^p\mathrm{Im}(H^{p+q}(Y_s)\xrightarrow{\pi_s^*} H^{p+q}(\hat Y_s)),$$
  where   $\pi_s:\hat Y_s\to Y_s$ is an arbitrary  projective resolution. 
\end{corollary}

\begin{remark}\label{rem-examples} 
 (1) The Du Bois complex for a (connected) curve $C$ was computed in \cite[Prop. 4.9]{duBois}, \cite[Example 7.23(2)]{PS}:  let $\tilde C$ be the normalization of $C$ and $C^w$ the weak normalization, so that there is a factorization   $\tilde C\to C^w\to C$.  Let $\nu\colon \tilde C \to C$ and $\nu'\colon C^w \to C$ be the corresponding morphisms. Then $\underline\Omega^0_C = \nu'_*\scrO_{C^w}$ and $\underline\Omega^1_C = \nu_*\Omega^1_{\tilde{C}}$. In particular,  $\uh^{1,0}(C)=p_g(C)=g(\tilde C)$, while $\uh^{0,1}(C)=p_a(C^w)$. Thus the only Hodge number  which is always  preserved in flat proper families of curves is  $\uh^{0,0}=1$. Note that $C$   has Du Bois singularities $\iff$ $C^w\cong C$, and $C$ has  rational singularities $\iff$ $C$ is smooth.   However, at least one of the Hodge symmetries $\uh^{1,0}\neq \uh^{0,1}=p_a(C)$ or $\uh^{1,1}\neq\uh^{0,0}=1$ will fail if the singularities of  $C$ are Du Bois but not rational  (depending on $C$ of compact/non-compact type), reflecting the fact that there are no rational singularities in dimension $1$. 
 
 \smallskip
 \noindent (2) Suppose that  $C$ has ordinary double points and let $\tau^\bullet_C$ be the subcomplex of ``torsion differentials" on $C$, i.e.\ those sections of $\Omega^\bullet_C$ which pull back to $0$ on $\tilde{C}$. By \cite[(1.1) and (1.5)]{F83}, $(\Omega^\bullet_C, d)$ and $(\Omega^\bullet_C/\tau^\bullet_C,d)$ are both  resolutions of $\Cee$. In fact, it is easy to see directly that (1) $(\Omega^\bullet_C/\tau^\bullet_C,d)$ is a resolution of $\Cee$ and (2) at an ordinary double point locally defined by $z_1z_2 =0$, $\tau^1_C =\Cee\cdot z_1\,dz_2$, $\tau^2_C = \Cee\cdot dz_1\wedge dz_2$, and $d\colon \tau^1_C \to \tau^2_C$ is an isomorphism. Hence $(\tau^\bullet_C,d)$ is acyclic and $(\Omega^\bullet_C, d)$ is quasi-isomorphic to  $(\Omega^\bullet_C/\tau^\bullet_C,d)$ and thus is a resolution of $\Cee$. In particular the spectral sequence with $E_1$ term   $H^q(C; \Omega^p_C) \implies \mathbb{H}^{p+q}(C; \Omega^\bullet_C) = H^{p+q}(C;\Cee)$  does not degenerate   at $E_1$ as $d=d_1$ is nonzero on $H^0(C; \tau^1_C) \subseteq H^0(\Omega^1_C)$.
 
 \smallskip
 \noindent (3) More generally, if $X$ has normal crossings, then  $\uOb_X \cong \Omega^\bullet_X/\tau^\bullet_X$ (with the trivial filtration), where as above $\tau^\bullet_X$ is the subcomplex of torsion differentials, by e.g.\ \cite[Example 7.23(1)]{PS}. Moreover, by \cite[(1.1)]{F83} and generalizing the explicit calculations of (2) above,  $(\Omega^\bullet_X, d)$ is also a resolution of $\Cee$ and in particular $(\Omega^\bullet_X, d)$ and $(\Omega^\bullet_X/\tau^\bullet_X, d)$ are quasi-isomorphic. For a related result, for a general variety $X$, Guill\'en showed in \cite[III(1.15)]{GNPP} that the homomorphism $\Omega^\bullet_X \to \uOb_X$ always has a section (viewed as a morphism in the derived category). However, these results do not seem to say much about the individual homomorphisms $\Omega^p_X \to \uOp_X$. 
 
 \smallskip
 \noindent (4)   Fourfolds $Y$ with  ADE hypersurface singularities (such as those occurring in \cite{Laza})  are examples of $1$-rational singularities. For such fourfolds,   Corollary \ref{Corkrat} implies that only $\uh^{2,2}(Y)$ can vary in small deformations. Thus,  any smoothing of $Y$ will have finite monodromy  (compare \cite[Cor. 1]{KLS}). 
\end{remark}

A brief description of the contents of this paper is as follows. Section~\ref{section2} deals with some basic results about K\"ahler differentials in the lci case. These include Theorem~\ref{flatness} regarding the flatness of the relative K\"ahler differentials and Proposition~\ref{KDexact} on restricting to a generic hypersurface section.  In Section~\ref{section3}, we give a quick review of the definition and the basic facts about higher Du Bois singularities (following \cite{MOPW}, \cite{JKSY-duBois}, and \cite[\S3]{FL}) and define higher rational singularities.  After these preliminaries, we establish Theorem~\ref{mainThm}. Our argument is close to the original argument (\cite[Lemma 1]{DBJ}) used to establish Theorem \ref{Thm-duBois},  following a suggestion of J. Koll\'ar. Finally, in Section~\ref{section5}, we prove Theorem \ref{thm-krat} following the strategy of \cite{Kovacs99}, and deduce  a consequence about the Hodge numbers of a smoothing  along the lines of Corollary~\ref{mainThmcor}. An appendix section by M. Saito discusses Conjecture~\ref{conjDBrat} in the hypersurface case.

Finally, beyond the conjectures and speculation we have already made, we emphasize the importance of extending these results wherever  possible to the non-lci case. Along these lines, Shen, Venkatesh and Vo have recently posted a preprint \cite{SVV} proposing different definitions of $k$-Du Bois and $k$-rational singularities which agree with the previous ones in the lci case. 

\subsection*{Notations and Conventions} We work in the complex algebraic category (but see Remark \ref{rem-can} for some possible generalizations).   Following our conventions from previous work \cite{FL}, we use $X$ (and $\cX$) in the local context (e.g. statements such as Theorem \ref{flatness} are purely local, and hold in the analytic case as well), whereas $Y$  and $\pi\colon \cY\to S$  are meant to be proper over $\Spec \Cee$ or $S$ respectively (e.g.\ in Theorem \ref{mainThm} where the properness is essential). The scheme or analytic space $X$ is an \textsl{algebraic variety} if it is  reduced and irreducible. In this case, its singular locus $X_{\text{\rm{sing}}}$ is denoted by $\Sigma$. We set $n =\dim X$ and $d =\dim \Sigma$.

\subsection*{Acknowledgement} We have benefited from discussions and correspondence  with J. Koll\'ar, M. Musta\c{t}\u{a}, and M. Popa on higher du Bois singularities while preparing \cite{FL}. The second author had several related discussions with M. Kerr and M. Saito while preparing previous joint work. We also thank J. de Jong, M. Saito and C. Schnell for some further comments related to this paper. After circulating an earlier version of this paper,  M. Musta\c{t}\u{a}, M. Popa informed us of some their recent work (\cite{MP-rat}, \cite{ChenDirksM}) related to Theorem \ref{thm-krat} and Conjecture~\ref{conjDBrat}. We are grateful to them for these communications.  M.  Saito kindly provided us with proofs of Proposition~\ref{hypersurfcase}(ii) and Conjecture~\ref{conjDBrat} in the hypersurface case, and agreed to include those as an appendix to our paper. Finally, we would like the referees for a very carefully reading of the first version of this paper and for many helpful comments.

\section{Some results on K\"ahler and relative K\"ahler differentials}\label{section2} The K\"ahler differentials are coherent sheaves that are determined by certain universal properties, including compatibility with base change. For smooth families $\cY/S$, the proof of the constancy of the Hodge numbers  uses in an essential way the semi-continuity of $h^q(\Omega^p_{Y_t})$, which in turn depends on the flatness of $\Omega^p_{\cY/S}$.  Here, we generalize this key point, noting that, for an lci morphism, $\Omega^p_{\mathcal X/S}$ is flat over $S$ (Theorem \ref{flatness}) for $p$ satisfying a bound depending on the dimension $d$ of the singular locus. Our argument depends essentially on the lci assumption, and it is a consequence of some depth estimates for $\Omega_X^p$ for $X$ with lci singularities due to Greuel.

A second result (Proposition \ref{KDexact}) that follows by related arguments are higher adjunction type results regarding restrictions of K\"ahler differentials to generic hypersurface sections. 
  
\subsection{A flatness result} We begin by recalling some basic notions concerning depth. Recall that, if $R$ is a local ring with maximal ideal $\mathfrak{m}$ and $M$ is an $R$-module, then $\operatorname{depth}M$ is the maximal length of a regular sequence for $M$ \cite[p.\ 120]{Matsumura}. If $\mathcal{F}$ is a coherent sheaf on a complex space $X$ and $x\in X$, let $\operatorname{depth}_x \mathcal{F} = \operatorname{depth} \mathcal{F}_x$, viewed as an $\scrO_{X,x}$-module. A key technical result we will need is then the following theorem  \cite{Scheja}, \cite[Satz 1.2]{Greuel75}:

\begin{theorem}[Scheja]\label{Schejathm} Suppose that $X$ is an analytic space, $A$ a closed analytic subspace, and $\mathcal{F}$ a coherent sheaf on $X$. Let $\rho\colon  H^0(X; \mathcal{F}) \to H^0(X-A; \mathcal{F}|X-A)$ be the restriction map. 
\begin{enumerate}
\item[\rm(i)] If $\operatorname{depth} \mathcal{F}_x \ge \dim A + 1$ for every $x\in X$, the homomorphism 
$\rho$ is injective. 
\item[\rm(ii)] If $\operatorname{depth} \mathcal{F}_x \ge \dim A + 2$ for every $x\in X$, the homomorphism $\rho$  is an isomorphism. \qed 
\end{enumerate}
\end{theorem}

It is well known that   K\"ahler $p$-differentials on singular spaces can have torsion. For instance this is already the case for $\Omega^1_C$ where $C$ is a nodal curve. However, we have the following \cite[Lemma 1.8]{Greuel75}

\begin{theorem}[Greuel]\label{Greuelthm} Suppose that $X$ is an lci singularity of dimension $n$ and that  $\dim \Sigma\leq d$. Then, for all $x\in X$, $\operatorname{depth}_x \Omega^p_X \geq n-p$ for $p\leq n-d$. More generally, let $f\colon \cX \to S$ be a flat lci morphism of relative dimension $n$ over a smooth base $S$  and let $\cX_{\text{\rm{crit}}}$ denote  the critical locus of $f$, i.e.\ the points of $\cX$ where $f$ is not a smooth morphism. If $\dim S = m$ and the relative dimension of $\cX_{\text{\rm{crit}}}$ is at most $d$ then, for every $x\in \cX$, then $\operatorname{depth}_x \Omega^p_{\cX/S} \geq \dim \cX -p = n+m-p$ for $p\leq n-d$. \qed
\end{theorem}

Before stating the next corollary, we fix the following notation which will be used for the rest of the section: If $f\colon \cX \to S$ is a morphism and $s\in S$, we denote by $X_s$  the fiber $f^{-1}(s)$ and by $\Sigma_s$ the singular locus of $X_s$: $\Sigma_s = (X_s)_{\text{\rm{sing}}}$.

\begin{corollary}\label{cor2} Suppose that $f\colon \cX \to S$ is a flat lci morphism of relative dimension $n$ over a smooth base $S$ of dimension $m$, and that  $\dim \Sigma_s\le d$, with $n-d\geq 2k+1$ for some integer $k\geq 1$ and every point $s\in S$.  Let $\cX_{\text{\rm{crit}}}$ denote the critical locus of $f$, i.e.\ the points of $\cX$ where $f$ is not a smooth morphism. Then, for every $x\in \cX$ and $p\leq n-d$, 
$$ \operatorname{depth}_x \Omega^p_{\cX/S} \geq d+m +2,$$
and, for all $p\leq k$ and every open subset $\cU$ of $\cX$, the restriction map
$$H^0(\cU; \Omega^p_{\cX/S}|\cU) \to H^0(\cU- \cX_{\text{\rm{crit}}}; \Omega^p_{\cX/S}|\cU- \cX_{\text{\rm{crit}}})$$
is an isomorphism.
\end{corollary} 
\begin{proof}  By assumption, $\dim \cX_{\text{\rm{crit}}} \leq d+m$. Note that $n-d \geq 2k+1 \geq k$, and hence, if $p\leq k$, then $p\leq n-d$. Thus Theorem~\ref{Greuelthm} implies that, for all $x\in \cX$,  
$$ \operatorname{depth}_x \Omega^p_{\cX/S} \geq n+m-p \geq n-k+m \geq d+m +k + 1 \geq d+m+ 2\geq \dim \cX_{\text{\rm{crit}}} +2. $$
Then $H^0(\cU; \Omega^p_{\cX/S}|\cU) \to H^0(\cU- \cX_{\text{\rm{crit}}}; \Omega^p_{\cX/S}|\cU- \cX_{\text{\rm{crit}}})$ is an isomorphism, by Theorem~\ref{Schejathm}(ii). 
\end{proof}

\begin{proposition}\label{prop4} Suppose that $f\colon \cX \to S$ is a flat lci morphism of relative dimension $n$ over a smooth base $S$ of dimension $m$, and that  $\dim \Sigma_s\leq d$, with $n-d\geq 2k+1$ for some integer $k\geq 1$ and every point $s\in S$. Suppose that $\cX \subseteq A\times S$, where $A$ is smooth, and let $I$ be the defining ideal of $\cX$ in $A\times S$. Then there exists a filtration of $\bigwedge ^k\pi_1^*\Omega_A^1|\cX = \pi_1^*\Omega_A^k|\cX$ by coherent subsheaves  $\mathcal{K}^a$, $0\leq a\leq k$, such that $\mathcal{K}^0 = \pi_1^*\Omega_A^k|\cX$,  $\operatorname{depth}_x\mathcal{K}^a \geq d+m+2$ for all $x\in \cX$, and, for all $a\geq 0$,  
$$\mathcal{K}^a/\mathcal{K}^{a+1} \cong \bigwedge ^a(I/I^2) \otimes  \Omega_{\cX/S}^{k-a}.$$
\end{proposition}
\begin{proof}The proof is by induction on $k$. The case $k=1$ is the conormal sequence
$$0 \to I/I^2 \to \pi_1^*\Omega_A^1|\cX \to \Omega_{\cX/S}^1 \to 0,$$
where we let $\mathcal{K}^1$ be the image of $I/I^2$. In general, for all $0\leq a\leq k$, let $\mathcal{K}^a$ be the image of $\bigwedge ^a(I/I^2) \otimes \pi_1^*\Omega_A^{k-a}|\cX$ in $\pi_1^*\Omega_A^k|\cX$. Thus $\{\mathcal{K}^a\}$ is the usual Koszul filtration, and the proposition is clear over all points of $\cX- \cX_{\text{\rm{crit}}}$. Moreover, $\mathcal{K}^k$ is either $0$ or $\bigwedge^k(I/I^2)$.   By hypothesis $k < n=\operatorname{rank}\Omega_{\cX/S}^1$. Let $r =\operatorname{rank}I/I^2$. We can clearly assume that $r\ge 1$.   If $r\geq k$, then $\mathcal{K}^k =\bigwedge^k(I/I^2)$. If $r< k$, then $\mathcal{K}^k =0$, and the largest $p$ such that $\mathcal{K}^p\neq 0$ is $\mathcal{K}^r$, the image of $\bigwedge^r(I/I^2) \otimes \pi_1^*\Omega_A^{k-r}|\cX$. In this case, the image of $\bigwedge^r(I/I^2) \otimes(I/I^2) \otimes \pi_1^*\Omega_A^{k-r-1}|\cX$ in  $\bigwedge^r(I/I^2) \otimes \pi_1^*\Omega_A^{k-r}|\cX$ maps to $0$ in $\pi_1^*\Omega_A^k|\cX$.  By the inductive hypothesis, the sequence 
$$\bigwedge^r(I/I^2) \otimes(I/I^2) \otimes \pi_1^*\Omega_A^{k-r-1}|\cX \to  \bigwedge^r(I/I^2) \otimes\pi_1^*\Omega_A^{k-r}|\cX \to \bigwedge^r(I/I^2) \otimes\Omega_{\cX/S}^{k-r} $$
is exact, so  there is an induced map $\varphi_r\colon \bigwedge^r(I/I^2) \otimes \Omega_{\cX/S}^{k-r}\to \pi_1^*\Omega_A^k|\cX$. Then the image of $\varphi_r$ is contained in the torsion free sheaf $\mathcal{K}^r$ and is equal to $\mathcal{K}^r$ over $\cX- \cX_{\text{\rm{crit}}}$. Moreover $\varphi_r$ is injective over $\cX- \cX_{\text{\rm{crit}}}$. By the inductive hypothesis on $k$, 
$$\operatorname{depth}_x\bigwedge^r(I/I^2) \otimes \Omega_{\cX/S}^{k-r} \geq d+m+2,$$ 
and hence $\im \varphi_r = \mathcal{K}^r$,  and  $\varphi_r \colon \bigwedge^r(I/I^2) \otimes \Omega_{\cX/S}^{k-r} \cong \mathcal{K}^r = \mathcal{K}^r/\mathcal{K}^{r+1}$ is an isomorphism. 

For a general $a$, there is a commutative diagram  
$$\begin{CD}
\scriptstyle 0 @>>> \scriptstyle\mathcal{J} @>>> \scriptstyle \bigwedge ^a(I/I^2) \otimes \pi_1^*\Omega_A^{k-a}|\cX @>>>\scriptstyle \bigwedge ^a(I/I^2) \otimes  \Omega_{\cX/S}^{k-a} @>>>
\scriptstyle0\\
@. @. @VVV @VV{\varphi_a}V @.\\
\scriptstyle0 @>>> \scriptstyle\mathcal{K}^{a+1} @>>> \scriptstyle\mathcal{K}^a @>>>\scriptstyle \mathcal{K}^a/\mathcal{K}^{a+1} @>>>\scriptstyle 0,
\end{CD}$$
where by definition $\mathcal{J} = \Ker\{\bigwedge ^a(I/I^2) \otimes \pi_1^*\Omega_A^{k-a}|\cX \to \bigwedge ^a(I/I^2) \otimes  \Omega_{\cX/S}^{k-a}\}$ and $\varphi_a$, which exists and is an isomorphism over $\cX - \cX_{\text{\rm{crit}}}$, is yet to be constructed over all of $\cX$.

For a fixed $k$, arguing by descending induction on $a$, and starting either at $k$ or at $r$, we can assume that  $\operatorname{depth}_x\mathcal{K}^{a+1} \geq d+m+2$. Moreover, $\mathcal{K}^{a+1}|\cX - \cX_{\text{\rm{crit}}}$ is a subbundle of $\mathcal{K}^a|\cX - \cX_{\text{\rm{crit}}}$, so the torsion subsheaf of $\mathcal{K}^a/\mathcal{K}^{a+1}$ is supported on $\cX_{\text{\rm{crit}}}$. If $\sigma$ is a section of $\mathcal{K}^a$ over some open subset $\cU$ and $\sigma|V \in \mathcal{K}^{a+1}$ for an open subset $\mathcal{V}=\cU-W$ of $\cU$, where $W$ is an analytic subset of $\cU$, then we may assume that $W = \cX_{\text{\rm{crit}}}$. Since $\mathcal{K}^a$ is torsion free and $\operatorname{depth}_x\mathcal{K}^{a+1} \geq d+m+2$, $\sigma \in \mathcal{K}^{a+1}(\cU)$. Thus $\mathcal{K}^a/\mathcal{K}^{a+1}$ is torsion free.

Let $\mathcal{I}$  be the kernel of the   map
$$\bigwedge ^a(I/I^2) \otimes \pi_1^*\Omega_A^{k-a}|\cX  \to \mathcal{K}^a/\mathcal{K}^{a+1}.$$
Since $\mathcal{K}^a/\mathcal{K}^{a+1}$ is torsion free, a standard argument shows that, for every open subset $\cU$ of $\cX$, the map 
$$H^0(\cU; \mathcal{I}|\cU) \to H^0(\cU- \cX_{\text{\rm{crit}}}; \mathcal{I}|\cU- \cX_{\text{\rm{crit}}})$$
is an isomorphism. 
Over $\cX - \cX_{\text{\rm{crit}}}$, there is an isomorphism $\mathcal{I}|\cX - \cX_{\text{\rm{crit}}} \cong  \mathcal{J}|\cX - \cX_{\text{\rm{crit}}}$. 
 Since$\bigwedge ^a(I/I^2) $ is locally free, $\operatorname{depth}_x\mathcal{J}\geq \operatorname{depth}_x\bigwedge ^a(I/I^2) \otimes  \Omega_{\cX/S}^{k-a} \geq d+m+2$ by Corollary~\ref{cor2}.   In particular,  
 $$H^0(\cU; \mathcal{J}|\cU) \to H^0(\cU- \cX_{\text{\rm{crit}}}; \mathcal{J}|\cU- \cX_{\text{\rm{crit}}})$$
is also an isomorphism.  Thus, $\mathcal{J}=\mathcal{I}$, and there is an induced injective homomorphism $\varphi_a$ as in the diagram.   Since $\varphi_a$ is an isomorphism over $\cX- \cX_{\text{\rm{crit}}}$ and  $\operatorname{depth}_x\bigwedge ^a(I/I^2) \otimes  \Omega_{\cX/S}^{k-a} \geq d+m+2$,  $\varphi_a$ is an isomorphism. Finally, since both $\operatorname{depth}_x\mathcal{K}^{a+1} \geq d+m+2$ and $\operatorname{depth}_x\mathcal{K}^a/\mathcal{K}^{a+1} \geq d+m+2$, $\operatorname{depth}_x\mathcal{K}^a \geq d+m+2$ as well.
\end{proof}

A key technical step in establishing Theorem~\ref{mainThm}, where the lci assumption seems crucial, is the following:  

\begin{theorem}\label{flatness}  Suppose that $f\colon \cX \to S$ is a flat lci morphism, where $S$ is an arbitrary  base space. Let $s\in S$ and suppose that $\codim_{X_s}  \Sigma_s \geq 2k+1$. Then, possibly after replacing $S$ by a neighborhood of $s$,  the sheaf of relative differentials $\Omega^p_{\cX/S}$ is flat over $S$ for all $p\leq k$. For $p= 0,1$, $\Omega^p_{\cX/S}$ is flat over $S$ with no assumption on the dimension of the singular locus.
\end{theorem}

\begin{remark} For $k\geq 2$, at least some assumption is needed for the conclusions of Theorem~\ref{flatness} to hold. For example, if $X$ has simple normal crossings singularities, say $X$ is the product of a nodal curve $\Spec \Cee[x,y]/(xy)$ with a smooth manifold of dimension $n-1$ and $\cX$ is the standard smoothing over $\mathbb{A}^1$ given by $xy=t$, then  $\operatorname{depth}_x \Omega^1_{\cX/\mathbb{A}^1} = n$ at a singular point and $\Omega^2_{\cX/\mathbb{A}^1}$ has torsion in the fiber over the singular fiber at $t=0$, hence is not flat over $\mathbb{A}^1$.

In the same spirit, with respect to Theorem~\ref{mainThm}, suppose that $\pi\colon \cC \to \Delta$ is a family of smooth projective curves acquiring a single ordinary double point over $0$, with the same local picture as above (so that $\cC$ is also smooth). Then   the sheaf  $\Omega^1_{\cC/\Delta}$ of relative K\"ahler differentials  is flat over $\Delta$ either by the case $p=1$ of Theorem~\ref{flatness} or by the direct computation that $\Omega^1_{\cC/\Delta}$ is isomorphic in a neighborhood of the singular point $p\in \cC$ to the maximal ideal $\frak{m}_p$, and in particular is torsion free over $\scrO_\Delta$. Let $C_t$ denote  the fiber over $t$.  If $C_0$ is reducible, the values of $h^0(\Omega^1_{C_t})$ and $h^1(\Omega^1_{C_t})$ jump up at $t=0$. Hence $R^1\pi_*\Omega^1_{\cC/\Delta}$ has torsion at $0$, and in particular is not locally free. This follows from very general results about cohomology and base change. In our situation, since $\Delta$ is smooth of dimension one, it follows directly by applying $R^i\pi_*$ to the exact sequence
$$0 \to \Omega^1_{\cC/\Delta} \xrightarrow{\times t}   \Omega^1_{\cC/\Delta} \to \Omega^1_{C_0} \to 0$$
to see that  the dimension of the $\scrO_\Delta$-module $R^1\pi_*\Omega^1_{\cC/\Delta}/tR^1\pi_*\Omega^1_{\cC/\Delta}$ is larger than the  rank of $R^1\pi_*\Omega^1_{\cC/\Delta}$.
\end{remark}

\begin{proof}[Proof of Theorem~\ref{flatness}]It suffices to consider the case $p=k$. Since every deformation is (locally) pulled back from the versal deformation, by standard properties of flatness and wedge product under base change we can assume that $S$ is smooth. First, with no assumption on the singular locus,   $\scrO_{\cX}$ is flat over $S$ by assumption. To see that $\Omega^1_{\cX/S}$ is flat over $S$, again with no assumption on the singular locus, we have   the conormal sequence
$$0\to I/I^2 \xrightarrow{u} \pi_1^*\Omega^1_A|\cX \to \Omega^1_{\cX/S} \to 0.$$
Let $I_s$ be the ideal of $X_s$ in $A$. By the lci assumption, $I_s = I \otimes \scrO_{S,s}/\mathfrak{m}_s$. Then the conormal sequence for $\cX/S$ becomes the corresponding sequence for $\Omega^1_{X_s}$ after tensoring with $\scrO_{S,s}/\mathfrak{m}_s$, namely
$$0\to I_s/I_s^2 \xrightarrow{\bar{u}}  \Omega^1_A|X_s  \to \Omega^1_{X_s} \to 0.$$
  In particular, $\bar{u}$ is injective. Hence $\Omega_{\cX/S}^1 =\operatorname{Coker} u$ is flat over $S$ by the local criterion of flatness \cite[(20.E) pp.\ 150--151]{Matsumura}. 

For $k\geq 2$, by induction on $k$ and descending induction on $a$, $\mathcal{K}^{a+1}$ and $\mathcal{K}^a/\mathcal{K}^{a+1}$ are flat over $S$ for all $a\geq 1$, and hence so is $\mathcal{K}^a$.  Let $\mathcal{K}^a_s$ be the image of $\bigwedge ^a(I_s/I_s^2) \otimes  \Omega_A^{k-a}|X_s$ in $ \Omega_A^k|X_s$.  Proposition~\ref{prop4} applied to the case $S =$ pt implies that $\mathcal{K}_s^a/\mathcal{K}_s^{a+1} \cong \bigwedge ^a(I_s/I_s^2) \otimes  \Omega_{X_s}^{k-a}$. We have a commutative diagram
$$\begin{CD}
\scriptstyle0 @>>> \scriptstyle\mathcal{K}^{a+1}\otimes (\scrO_{S,s}/\mathfrak{m}_s) @>>> \scriptstyle\mathcal{K}^a\otimes (\scrO_{S,s}/\mathfrak{m}_s) @>>> \scriptstyle \mathcal{K}^a/\mathcal{K}^{a+1}\otimes (\scrO_{S,s}/\mathfrak{m}_s) @>>> \scriptstyle 0\\
@. @VVV @VVV @VV{\cong}V @.\\
\scriptstyle0 @>>> \scriptstyle\mathcal{K}_s^{a+1} @>>> \scriptstyle\mathcal{K}_s^a @>>> \scriptstyle \mathcal{K}_s^a/\mathcal{K}_s^{a+1} @>>> \scriptstyle 0
\end{CD}$$
Here, the top row is exact for $a\geq 1$ by the flatness of $\mathcal{K}^a/\mathcal{K}^{a+1}$ and induction on $k$. 
By descending induction on $a$, the above diagram implies
that the natural map $\mathcal{K}^a\otimes (\scrO_{S,s}/\mathfrak{m}_s)\to \mathcal{K}_s^a$ is an isomorphism for all $a\geq 1$. In particular, for $a=1$, we have the map $u\colon \mathcal{K}^1 \to \pi_1^*\Omega_A^k|\cX$ with cokernel $\Omega_{\cX/S}^k$, and the reduction $\bar{u}$ of $u$ mod $\mathfrak{m}_s$ is the natural inclusion $\mathcal{K}_s^1 \to \Omega_A^k|X_s$. Since $\bar{u}$ is injective,  $\Omega_{\cX/S}^k =\operatorname{Coker} u$ is flat over $S$ by the local criterion of flatness as above. 
This completes the proof of Theorem~\ref{flatness}.
\end{proof}

\subsection{K\"ahler differentials and generic hyperplane sections}\label{S-Bertini}

\begin{proposition}\label{KDexact} Let $X$ be a reduced local complete intersection with $X_{\text{\rm{sing}}}= \Sigma$ and let $H$ be a Cartier divisor on $X$ which is a generic element of a base point free linear system, with ideal sheaf $\scrO_X(-H)$. If $\codim \Sigma \ge 2k-1$, then, setting $\scrO_H(-H) =\scrO_X(-H)|H$, we have 
\begin{equation*}
0 \to \Omega^{p-1}_H \otimes  \scrO_H(-H) \to \Omega^p_X|H \to \Omega^p_H \to 0.\tag*{$(*)_p$}
\end{equation*}
 is  exact for $p \le k$.
 \end{proposition} 
\begin{proof} We can assume that $X$ is affine and that $H =(f)$ for a function $f\colon X \to \mathbb{A}^1$ which has no critical points away from $\Sigma$. 
The sequence $(*)_p$ is always exact for $p=0$,  and it is for $p=1$ because we have the exact sequence
$$\scrO_H(-H)   \to \Omega^1_X|H \to \Omega^1_H \to 0,$$
with $\scrO_H(-H) \cong \scrO_H$ locally free and $H$ reduced (as it is generically reduced and lci).  

By induction on $k$, it suffices to consider the case $p=k$. By the remarks above, we can assume $k\ge 2$. By the de Rham lemma   \cite[Lemma 1.6]{Greuel75}, we have the following:

\begin{lemma} Let $d =\dim \Sigma$. For $0\le p < n-d$, the following sequence is exact:
$$0 \to \scrO_X \xrightarrow{df \wedge} \Omega^1_X \xrightarrow{df \wedge} \cdots \xrightarrow{df \wedge} \Omega^{p+1}_X .$$
Equivalently, for $0\le p < n-d$,
$$df \wedge \colon \Omega^p_X/ df \wedge \Omega^{p-1}_X \to \Omega^{p+1}_X$$
is injective.  
\end{lemma} 
\begin{proof} Since $f$ is generic, the critical set $C(f)$ of $f$ (in the notation of \cite{Greuel75}) is just $\Sigma$. Then the result follows immediately from \cite[Lemma 1.6]{Greuel75} (with $g=h$, $k=1$, and $C(f,g)\cap N(h) = \Sigma$). 
\end{proof}

From the lemma, since $n-d = \codim \Sigma \ge 2k-1> k $, because $k >1$, there is an exact sequence
$$ 0 \to \Omega^{k-1}_X/ df \wedge \Omega^{k-2}_X \xrightarrow{df \wedge} \Omega^k_X \to \Omega^k_X/ df \wedge \Omega^{k-1}_X \to 0.$$
 Tensoring with $\scrO_H$ gives an exact sequence
 $$0 \to \mathit{Tor}_1^{\scrO_X} (\Omega^k_X/ df \wedge \Omega^{k-1}_X, \scrO_H) \to \Omega^{k-1}_H \to \Omega^k_X|H \to \Omega^k_H \to 0.$$
 So we have to show that, for a general choice of $f$,  $\mathit{Tor}_1^{\scrO_X} (\Omega^k_X/ df \wedge \Omega^{k-1}_X, \scrO_H)  = 0$. Since $\scrO_H$ has the free resolution 
 $$0 \to \scrO_X \xrightarrow{\times f}  \scrO_X \to \scrO_H \to 0,$$
 it suffices to show that $\Omega^k_X/ df \wedge \Omega^{k-1}_X$ has no $f$-torsion, i.e.\ that multiplication by $f$ is injective. Since there is an inclusion $\Omega^k_X/ df \wedge \Omega^{k-1}_X \to \Omega^{k+1}_X$, it suffices to prove that $\Omega^{k+1}_X$ has no $f$-torsion. Let $\tau$ be an  $f$-torsion local section   of $\Omega^{k+1}_X$.  Note that the support  of   $\tau$   must be contained in $\Sigma \cap V(f)$, since $\Omega^{k+1}_X|X-\Sigma$ is torsion free and $f$ is invertible on $X- V(f)$. By the assumption that $H$  is general, every component of $\Sigma \cap V(f)$ has dimension $\le d-1$. By  \cite[Lemma 1.8]{Greuel75}, since $k+1 \le 2k-1 \leq \codim \Sigma$ as long as $k\ge 2$, we have, for every $x\in X$, 
 $$\operatorname{depth}_x\Omega^{k+1}_X \geq n - (k+1) .$$
 Using $n-d \geq 2k-1$ gives $n-k-1 \geq d+ k -2 \geq d$, since we are assuming $k\geq 2$. Thus 
 $$\operatorname{depth}_x\Omega^{k+1}_X \geq \dim (\Sigma \cap V(f)) + 1.$$
 It follows from Theorem~\ref{Schejathm}(i)  that, for every open subset $U$ of $X$,  the restriction map 
 $$H^0(U; \Omega^{k+1}_X|U) \to H^0(U-(\Sigma \cap V(f); \Omega^{k+1}_X|U-(\Sigma \cap V(f))$$
 is injective. In particular,   $\tau=0$. Hence $(*)_p$ is   exact for $p \le k$. 
\end{proof}

Note that we showed the following in the course of proving Proposition~\ref{KDexact}:

\begin{corollary}\label{KDexactcor} With $X$ and $H$ as in the statement of Proposition~\ref{KDexact}, 
$\mathit{Tor}_1^{\scrO_X} (\Omega^p_X, \scrO_H)= 0$ for all $p\le k+1$. \qed
\end{corollary}

\section{$k$-Du Bois and $k$-rational singularities}\label{section3}
We now review the notion of Du Bois and higher Du Bois singularities, and propose a general definition for higher rational singularities. After verifying that our definition agrees with the standard definition of rational singularities, as well as with previous definitions of higher rational singularities (under mild assumptions, expected to hold in general), we discuss the connection between higher rational and higher Du Bois singularities.  

\subsection{The filtered de Rham complex} 
On a smooth scheme $Y$ proper over $\Spec \Cee$ or a compact K\"ahler manifold, the Hodge-de Rham spectral sequence with $E_1$ page $E_1^{p,q}= H^q(Y;\Omega^p_Y)$ degenerates at $E_1$ and computes the Hodge structure on $H^{p+q}(Y,\Cee)$. For $X$ not necessarily smooth or proper over $\Spec \Cee$,  Deligne showed that $H^*(X,\Cee)$ carries a canonical mixed Hodge structure. Subsequently,  Du Bois \cite{duBois} introduced an object $\uOb_X = (\uOb_X, F^\bullet\uOb_X)$ in the filtered derived category whose graded pieces $\uOp_X=\Gr^p_F\uOb_X[p]$ are analogous to   $\Omega_X^p$ in the smooth case (see \cite[\S7.3]{PS}). Since $\uOp_X$ is defined locally  in the \'etale topology, it agrees with $\Omega_X^p$ at smooth points. Near the singular locus, if $\pi:\hX\to X$ is a log resolution with $E\subseteq \hX$ the reduced exceptional divisor, $\uOb_X$ is closely related to  $R\pi_*\Omega^\bullet_{\hX}(\log E)(-E)$.   More precisely, we have the following by \cite[Example 7.25]{PS} and by  \cite[\href{https://stacks.math.columbia.edu/tag/05S3}{Tag 05S3}]{stacks-project}:  

\begin{theorem}\label{relative} If $\pi\colon \hX \to X$ is a log resolution with reduced exceptional divisor $E$, and we define
    $\uOb_{X,\Sigma} = R\pi_*\Omega^\bullet_{\hX}(\log E)(-E)$, where $\Sigma$ is the singular locus of $X$, then there is the distinguished triangle of relative cohomology in the filtered derived category
$$\uOb_{X,\Sigma} \to \uOb_X \to \uOb_\Sigma \xrightarrow{+1} .$$
Thus, in the derived category, there is a corresponding distinguished  triangle  
$$\uOp_{X,\Sigma} \to \uOp_X \to \uOp_\Sigma \xrightarrow{+1} .\qed$$
\end{theorem}

In the proper case, there is the following fundamental result of Du Bois, based on Deligne's construction of the mixed Hodge structure on $Y$:

\begin{theorem}[Du Bois]\label{thm-e1} If $Y$ is proper over $\Spec \Cee$,  then  the spectral sequence with $E_1$ page $\mathbb{H}^q(Y; \underline{\Omega}^p_Y) \implies  \mathbb{H}^{p+q}(Y;\underline{\Omega}^\bullet_Y) = H^{p+q}(Y; \Cee)$ degenerates at $E_1$ and the corresponding filtration on $H^*(Y;\Cee)$ is the Hodge filtration associated to the mixed Hodge structure on $H^*(Y;\Cee)$. \qed
\end{theorem}

The result above allows us to define Hodge numbers in the singular case: 
\begin{definition}\label{defHodgeno}
Let $Y$ be a compact complex algebraic variety. We define the {\it Hodge-Du Bois numbers} 
$$\uh^{p,q}(Y):=\dim\Gr^p_FH^{p+q}(Y),$$
for $0\le p,q\le n$. In particular, if $Y$ is smooth, $\uh^{p,q}(Y)=h^{p,q}(Y)$. In general, 
by Theorem \ref{thm-e1}, $$\uh^{p,q}(Y)=\dim \mathbb H^q(Y; \uOp_Y)=\sum_{0\le r\le q} h^{p,r}_{p+q}(Y),$$ where $h^{p,r}_{p+q}(Y):=\dim \Gr^p_F\Gr^W_{p+r} H^{p+q}(Y)$ are the {\it Hodge-Deligne numbers} associated to the mixed Hodge structure on $Y$. 
\end{definition}

\begin{remark}
As the example of nodal curves shows (Remark~\ref{rem-examples}) the Hodge-Du Bois diamond will not satisfy either of the Hodge symmetries: $\uh^{p,q}=\uh^{q,p}$ or $\uh^{p,q}=\uh^{n-p,n-q}$. Nonetheless, some vestige  of these symmetries remains (in the form of inequalities) as those given by Lemma \ref{ineq-sym} below.
\end{remark}
Another key consequence for us is the following:

\begin{corollary}\label{DBdegen} With notation as above,
\begin{enumerate} 
\item[\rm(i)] The natural map $H^i(Y; \Cee) \to \mathbb{H}^i(Y;\underline{\Omega}^\bullet_Y/F^{k+1}\underline{\Omega}^\bullet_Y)$ is surjective for all $i$ and $k$.
\item[\rm(ii)] The spectral sequence with $E_1$ term
$$E_1^{p,q} = \begin{cases} \mathbb{H}^q(Y; \underline{\Omega}^p_Y) , &\text{for $p\leq k$;}\\
0, &\text{for $p> k$.}
\end{cases}$$
converging to $\mathbb{H}^{p+q}(Y; \underline{\Omega}^\bullet_Y/F^{k+1}\underline{\Omega}^\bullet_Y)$ degenerates at $E_1$. \qed
\end{enumerate}
\end{corollary} 
\begin{proof} This is a consequence of the following general fact: If $(\mathcal{C}^\bullet, d,\mathcal{F}^\bullet\mathcal{C}^\bullet)$ is a filtered complex, then the associated spectral sequence degenerates at the $E_1$ page $\iff$ the differential $d$ is strict with respect to the filtration $\mathcal{F}^\bullet\mathcal{C}^\bullet$. Then, assuming $E_1$-degeneration, an easy argument shows that $H^i(\mathcal{C}^\bullet,d) \to H^i(\mathcal{C}^\bullet/\mathcal{F}^{k+1}\mathcal{C}^\bullet,d)$ is surjective for every $i$ and $k$. Since the induced filtration on $\mathcal{C}^\bullet/\mathcal{F}^{k+1}\mathcal{C}^\bullet$ is also strict with respect to $d$, the corresponding spectral sequence for the complex $\mathcal{C}^\bullet/\mathcal{F}^{k+1}\mathcal{C}^\bullet$ also degenerates at $E_1$.
\end{proof}

\subsection{Du Bois and higher Du Bois singularities} There is a natural comparison map 
$$\phi^\bullet \colon (\Omega_X^\bullet, \sigma^\bullet)\to (\uOb_X, F^\bullet\uOb_X)$$    \cite[p.175]{PS}, where $\sigma^\bullet$ is the trivial or naive filtration on $\Omega_X^\bullet$. Following \cite{MOPW} and \cite{JKSY-duBois}, one defines:

\begin{definition}
Let $X$ be a complex algebraic variety. Then $X$ is \textsl{$k$-Du Bois} if the natural maps
$$\phi^p:\Omega_X^p\to \uOp_X$$
are quasi-isomorphisms for $0\le p\le k$. Note that the case $k=0$ coincides with the usual definition of Du Bois singularities  \cite[Def. 7.34]{PS}.
\end{definition} 

\begin{example} Let $f(z_1, \dots, z_{n+1}) = z_1^{d_1} + \cdots + z_{n+1}^{d_{n+1}}$ define the weighted homogeneous singularity $X =V(f) \subseteq \mathbb{A}^{n+1}$. Then, as a consequence of Proposition~\ref{hypersurfcase}(i) and a theorem of Saito \cite[(2.5.1)]{SaitoV} (and see also \cite[Corollary 6.8]{FL22d}), $X$ is $k$-Du Bois at $0$ 
$\iff$  $\Dis\sum_{i=1}^{n+1} \frac1{d_i} \geq k+1$. 
In particular, an  ordinary double point of dimension $n$ is $k$-Du Bois for all $k\leq  \Dis \left[\frac{n-1}{2}\right]$. Thus an ordinary double point of dimension $3$ is $1$-Du Bois, and in fact is the unique $1$-Du Bois hypersurface singularity in dimension $3$ (e.g. \cite[Thm. 2.2]{NS}).  More generally, it is expected that $k$-Du Bois singularities occur first in dimension $2k+1$. This is true at least in the lci case as the following result shows. 
\end{example}

\begin{theorem}[{\cite{MP-loc}, \cite{MOPW}, \cite{JKSY-duBois}}]\label{thmdim}
Let $X$ be a complex algebraic variety with lci singularities. Assume $X$ is $k$-Du Bois with $k\ge 1$. Then $X$ is normal and regular in codimension $2k$,   i.e.\ $\codim \Sigma \ge 2k+1$. 
\end{theorem}
\begin{proof}
The normality claim is \cite[Cor. 5.6]{MP-loc}. For hypersurface singularities, various dimension bounds (covering the claim of the theorem) were obtained in both \cite{MOPW} and \cite{JKSY-duBois}. The general lci case follows from \cite[Thm. F]{MP-loc} (numerical characterization of $k$-Du Bois) and \cite[Cor. 3.40]{MP-loc} (bounds on $\dim \Sigma$ in terms of the relevant numerical invariant). 
\end{proof}

 \begin{corollary}\label{cor1} Suppose that $X$   has  lci $k$-Du Bois singularities and that   $f\colon \cX \to S$ is a flat morphism, where $S$ is arbitrary, with $X=X_s=f^{-1}(s)$ for some $s\in S$. Then possibly after replacing $S$ by a neighborhood of $s$,  the sheaf of relative differentials $\Omega^p_{\cX/S}$ is flat over $S$ for all $p\leq k$.
\end{corollary} 
\begin{proof} This is immediate  from Theorem~\ref{flatness} and Theorem~\ref{thmdim}. 
\end{proof}
 
\begin{remark} (i)   Theorem~\ref{thmdim} and  Theorem~\ref{Greuelthm} imply the following previously known results:
\begin{enumerate}
\item[(1)] If $X$ has hypersurface $k$-Du Bois singularities, then $\Omega_X^p$ is torsion free for $0\le p\le k$ (\cite[Prop. 2.2]{JKSY-duBois}).
\item[(2)]  If  $X$ has lci $k$-Du Bois singularities, then $\Omega_X^p$ is reflexive for $0\le p\le k$ (\cite[Cor. 5.6]{MP-loc}).  
\end{enumerate}

\smallskip
\noindent (ii) For $p=1$ and lci singularities, the situation is well understood by a result of Kunz   \cite[Prop. 9.7]{Kunz}:   $\Omega_X^1$ satisfies Serre's condition $S_a$ $\iff$ $X$ satisfies $R_a$ (i.e.\ is regular in codimension $a$).   
\end{remark}

\begin{remark} For other examples of  $k$-Du Bois and $k$-rational singularities, one can consult \cite{KL2} (cf.\ Theorem 5.3 and \S6.1).
\end{remark}

\subsection{Higher rational singularities} The standard definition of a rational singularity involves the choice of a resolution (e.g. \cite[Def. 5.8]{KollarMori}). As we will explain below, it is possible to give an equivalent definition for rational singularities without reference to resolutions, but using instead the dualizing complex. In addition to being more intrinsic, it generalizes to higher rational singularities, and it factors naturally through the higher Du Bois condition.

 For a complex algebraic variety $X$ of dimension $n$, let $\omega^\bullet_X$ denote the dualizing complex. Define the    {\it Grothendieck duality functor} $\mathbb D_X$ as follows:
$$\mathbb D_X(-):= RHom(-, \omega_X^\bullet)[-n].$$
 In particular, $\omega_X^\bullet=\mathbb D_X(\scrO_X)[n]$.

\begin{lemma}\label{lem-psi} For all $p$, there exists a natural sequence of maps in the derived category
$$\Omega_X^p\xrightarrow{\phi^p} \uOp_X\xrightarrow{\psi^p}  \bD_X(\uOnp_{X}).$$
\end{lemma}
\begin{proof} By functoriality of the filtered de Rham complex, there is a map $\uOnp_X\to R\pi_*\Omega^{n-p}_{\hX}$.
Applying $\bD_X$ gives
$$\bD_X(R\pi_*\Omega^{n-p}_{\hX})\to \bD_X(\uOnp_X).$$
Since $\pi$ is proper, Grothendieck duality gives
$$\bD_X(R\pi_*\Omega^{n-p}_{\hX})\cong R\pi_*\bD_{\hX}(\Omega^{n-p}_{\hX})\cong R\pi_*\Omega_{\hX}^p.$$
 Thus we get a sequence of maps
 $$\Omega_X^p\to \uOp_X\to R\pi_*\Omega_{\hX}^p \to \bD_X(R\pi_*\Omega^{n-p}_{\hX}) \to \bD_X(\uOnp_X),$$
 as claimed. The map    $\psi^p$ is easily seen to be independent of the choice of a resolution, by the usual factorization arguments.
\end{proof}

\begin{definition}\label{defkrat} The variety $X$ has \textsl{$k$-rational singularities} if the maps 
$$\Omega_X^p\xrightarrow{\psi^p \circ \phi^p}  \bD_X(\uOnp_{X})$$ 
are quasi-isomorphisms for all $0\le p\le k$. 
\end{definition}

\begin{example} Let $f(z_1, \dots, z_{n+1}) = z_1^{d_1} + \cdots + z_{n+1}^{d_{n+1}}$ define the weighted homogeneous singularity $X =V(f) \subseteq \mathbb{A}^{n+1}$. Then, by \cite[Corollary 6.8]{FL22d}, $X$ is $k$-rational at $0$ $\iff$
 $\Dis\sum_{i=1}^{n+1} \frac1{d_i} > k+1$. 
In particular, an  ordinary double point of dimension $n$ is $k$-rational $\iff$ $k< \Dis  \frac{n-1}{2}$. Thus an ordinary double point of dimension $3$ is not $1$-rational.  On the other hand, ADE singularities in dimension $4$ are $1$-rational. Conjecturally, $k$-rational singularities occur in codimension at least $2(k+1)$ (compare Theorem~\ref{thmdim}). 
\end{example}

 The following lemma connects our definition to more standard ones:
\begin{lemma}\label{lemma-onp} Suppose that $\dim \Sigma \leq  d$. Then, for all $p< n-d$, 
$$\bD_X(\uOnp_{X}) \cong R\pi_*\Omega^p_{\hX}(\log E).$$
In particular, if $\codim \Sigma \geq 2k+1$ for some $k\geq 0$, then $\bD_X(\uOnp_{X}) \cong R\pi_*\Omega^p_{\hX}(\log E)$ for all $p\leq k$. 
\end{lemma}
\begin{proof} By Theorem~\ref{relative}, there is the distinguished triangle of relative cohomology
$$\uOnp_{X,\Sigma} \to \uOnp_X \to \uOnp_\Sigma \to . $$
Since $n-p > d$, $\uOnp_\Sigma =0$ and hence $\uOnp_X \cong \uOnp_{X,\Sigma}= R\pi_*\Omega^{n-p}_{\hX}(\log E)(-E)$.  Applying Grothendieck duality, it follows as in \cite[\S2.2]{MOPW} that 
\begin{align*}
\bD_X(\uOnp_{X}) &= \bD_X(R\pi_*\Omega^{n-p}_{\hX}(\log E)(-E)) = R\pi_*\bD_{\hX}(\Omega^{n-p}_{\hX}(\log E)(-E)) \\
&=  R\pi_*\Omega^p_{\hX}(\log E).
\end{align*}
The final statement is clear since $n-d\ge 2k+1$ $\implies$ $n-p \geq n-k \geq d+k+1 > d$.
\end{proof}

\begin{corollary}\label{cor-rat0} $X$ is $0$-rational $\iff$ $X$ has rational singularities. 
\end{corollary} 
\begin{proof} Since $X$ is reduced, $\dim \Sigma \leq  n-1$. Thus  $X$ is $0$-rational $\iff$ the natural map $\scrO_X \to R\pi_*\scrO_{\hX}$ is an isomorphism $\iff$ $X$ has rational singularities in the usual sense.
 \end{proof}
 
 \begin{remark} 
 Lemma \ref{lemma-onp} and Grothendieck duality give the identification $\underline{\Omega}_X^n=R\pi_*\omega_{\hX}$, which by Grauert-Riemenschneider vanishing is in fact a single sheaf, the {\it Grauert-Riemenschneider sheaf} $\omega_X^{GR}:=\pi_*\omega_{\hX}$. Following \cite{KK20}, let us denote by $\underline{\omega}_X:=\bD_X(\underline{\Omega}_X^0)$. Then the dual form of Definition \ref{defkrat} for $k=0$ is: $X$ has rational singularities $\iff$ the composite map
\begin{equation}\omega_X^{GR}\to\underline{\omega}_X\to \omega_X^\bullet
\end{equation}
 is a quasi-isomorphism. This formulation occurs for instance in \cite{KK20}, and it is equivalent to that given by \cite[Thm. 5.10(3)]{KollarMori} (note that the quasi-isomorphism $\omega_X^{GR}\cong \omega_X^\bullet$ forces $X$ to be Cohen-Macaulay). 
 \end{remark}

In \cite[\S3]{FL}, we  defined $k$-rational singularities for an isolated singularity by the condition that $R\pi_*\Omega^p_{\hX}(\log E)\cong \Omega_X^p$. 
This is  equivalent to Definition \ref{defkrat} (under a mild assumption):
\begin{corollary}\label{cor-olddef} In the above notation, suppose that  $\codim \Sigma \geq 2k+1$.
 \begin{enumerate}
\item[\rm(i)]  $X$ is $k$-rational $\iff$ the natural map $\Omega_X^p \to R\pi_*\Omega^p_{\hX}(\log E)$ is an isomorphism for all $p\le k$. 
\item[\rm(ii)]   $X$ is $k$-rational $\iff$  $X$ is $(k-1)$-rational and $\Omega_X^k \to   R\pi_*\Omega^k_{\hX}(\log E)$ is an isomorphism. \qed
\end{enumerate} 
 \end{corollary}

In the lci case, the assumption on $R^0\pi_*\Omega^p_{\hX}(\log E)$ is automatic:

\begin{lemma}\label{case0krat} Suppose that $X$ has lci singularities and $\codim \Sigma \geq 2k+1$. Then, for all $p\leq k$,  $\psi^p\circ \phi^p\colon \Omega_X^p \to R^0\pi_*\Omega^p_{\hX}(\log E)$ is an isomorphism. Hence  $X$ is $k$-rational $\iff$ for all $p\le k$ and all $q> 0$, $R^q\pi_*\Omega^p_{\hX}(\log E) =0$.
 \end{lemma} 
 \begin{proof} This follows easily from Theorem~\ref{Greuelthm} and  Corollary~\ref{cor2}, as $R^0\pi_*\Omega^p_{\hX}(\log E)$ is torsion free for all $p$.
 \end{proof}

  \begin{remark}\label{case0kDB}  Suppose that $X$ has an isolated singularity $x$, so that $\Sigma = \{x\}$. Then by Theorem~\ref{relative}, there is the distinguished triangle of relative cohomology
$$\uOb_{ X,x} \to \uOb_X \to \Cee[0] \to , $$
where we somewhat carelessly write $\Cee[0]$ for the skyscraper sheaf $\Cee_x$, viewed as a complex in degree $0$.
Hence $\uOp_{X,x} \to \uOp_X$ is an isomorphism for $p>0$, so if $X$ is lci and $n = \dim X \ge 2k+1$, by the same argument as that of Lemma~\ref{case0krat}, $ \phi^p\colon \Omega_X^p \to R^0\pi_*\Omega^p_{\hX}(\log E)(-E)$ is an isomorphism.  For $p=0$, if $X$ has an isolated singularity and is normal, then the map $\pi_*\scrO_{\hX}(-E) \to \underline{\Omega}^0_X$ has cokernel $\Cee[0]$ and factors through $\mathfrak{m}_x\subseteq \scrO_X$. Hence $\scrO_X \to \mathcal{H}^0\underline{\Omega}^0_X$ is an isomorphism. Thus, if $X$ has an isolated lci singularity and $\dim X   \ge 2k+1$,  i.e.\ $k \le (n-1)/2$,  then $X$ is $k$-Du Bois $\iff$ for all $p\le k$ and all $q> 0$, $R^q\pi_*\Omega^p_{\hX}(\log E)(-E) =0$.
 \end{remark}
 
 \subsection{$k$-rational vs.\ $k$-Du Bois singularities}
Steenbrink \cite{Steenbrink} proved that, if $X$ is an isolated rational singularity, then $X$ is Du Bois, and Kov\'acs \cite{Kovacs99} generalized this, showing that that any rational singularity is Du Bois. Saito gave a different proof in \cite[Thm.\ 5.4]{Saito2000}.  The method of \cite{Steenbrink} generalizes to prove that isolated $k$-rational singularities are $k$-Du Bois (see \cite[Theorem 3.2]{FL22d}), and the method  of \cite{Kovacs99} can be generalized to handle both isolated and lci singularities.  Using the ideas of \cite{Kovacs99},  we shall show the following in Section~\ref{section5}:

\begin{theorem}\label{kratkDB} Suppose either that $X$ has isolated singularities or  $X$ is lci.  If $X$ is $k$-rational, then $X$ is $k$-Du Bois. 
\end{theorem}

 Assuming for the moment Theorem \ref{kratkDB}, we note that the $k$-rational assumption adds additional 
 Hodge symmetries  to the $k$-Du Bois condition:  If $X$ is $k$-Du Bois, then $X$ is $k$-rational $\iff$ the map $\psi^p \colon \uOp_X\to  \bD_X(\uOnp_{X})$ is a quasi-isomorphism for all $p\le k$. If $X$ is proper, these assumptions  lead  to a Hodge symmetry,  Serre  duality for the Hodge-Du Bois numbers:
\begin{corollary}\label{cor-sym1} 
If $Y$ is a compact complex algebraic variety of dimension $n$ with lci $k$-rational singularities, then, for $0\le p\le k$, 
\begin{equation*}
\uh^{p,q}(Y)=\dim \Gr_F^p H^{p+q}(Y) = \dim \Gr_F^{n-p} H^{2n-(p+q)}(Y)=\uh^{n-p,n-q}(Y).
\end{equation*}
 
In particular,  taking  $p+q =n$   gives   $\uh^{p,n-p}=\uh^{n-p,p}$ for  $p\le k$. 
\end{corollary}
\begin{proof} Using Theorem \ref{thm-e1}, we get the 
following identifications
$$\Gr_F^p H^{p+q}(Y) \cong \mathbb{H}^q(Y; \uOp_Y)\cong\mathbb{H}^q(Y; \bD_Y(\uOnp_{Y}))\cong \mathbb{H}^{n-q}(Y; \uOnp_{Y})^\vee,$$
where the middle isomorphism is given by the quasi-isomorphism $\psi^p$ (for $p\le k$). \end{proof}

It remains to discuss the Hodge symmetry $h^{p,q}=h^{q,p}$, which is induced by   complex conjugation in the smooth case. We recall that the cohomology of a compact singular algebraic variety $Y$ carries a mixed Hodge structure $(H^*(Y),F^\bullet, W_\bullet)$. The Hodge-Deligne numbers $h^{p,r}_i=\Gr^p_F\Gr_{p+r}^WH^i(Y)$  satisfy the symmetry given by  conjugation: $h^{p,r}_i=h^{r,p}_i$ (as $\Gr_{p+r}^WH^i(Y)$ is a pure Hodge structure). Since $Y$ is compact,  the weights on $H^i(Y)$ are at most $i$, and in fact between $2i-2n$ and $i$ if $i\geq n$. It follows that, for $i\le n$, 
$\uh^{p,i-p}=\sum_{r=0}^{i-p} h^{p,r}_i$. However, the Hodge-Du Bois numbers do not satisfy the same kind of symmetry as they reflect only the Hodge filtration $F^\bullet$. In fact, we note the following:
\begin{lemma}\label{ineq-sym}
Let $Y$ be a compact complex algebraic variety of dimension $n$.  
For $0\le p\le i\le n$, 
$$ \sum_{a=0}^p \uh^{i-a,a} \le 
 \sum_{a=0}^p \uh^{a,i-a}.$$
  Furthermore, equality holds above for all $p\le k$ $\iff$ $\uh^{p,i-p}=\uh^{i-p,p}$ for all $p\le k$   $\iff$ $\Gr_F^pW_{i-1}H^i(Y)=0$ for all $p\le k$. 
\end{lemma}
\begin{proof} 
Clearly $\sum_{a=0}^p \uh^{i-a,a} = 
 \sum_{a=0}^p \uh^{a,i-a}$ for all $p\le k$ $\iff$ $\uh^{p,i-p}=\uh^{i-p,p}$ for all $p\le k$. 

For a fixed $p$, we have
 $$\sum_{a=0}^p \uh^{i-a,a} = \sum_{a=0}^p \sum_{q\leq a}h^{i-a, q}_i =  \sum_{a=0}^p \sum_{q\leq a}h^{q,i-a}_i=\sum_{\substack {i-p \le s \le i \\r+s \le i}}h^{r, s}_i,$$
by the Hodge symmetries,  whereas
 $$\sum_{a=0}^p \uh^{a,i-a} = \sum_{a=0}^p \sum_{a+q \le i}h^{a, q}_i=\sum_{\substack { r \le p \\r+s \le i}}h^{r, s}_i.$$
Note that, if  $i-p \le s \le i $ and $r+s \le i$, then $r\le i-s \le p$, so the second sum is greater that  the first, giving the inequality. 
For a given $p$, equality  holds $\iff$ $h^{r, s}_i = 0$ for $r\leq p$ and $s\leq i-p -1$. Moreover, $h^{r, s}_i = 0$ for $r\leq p$ and $s\leq i-p -1$ for some $p\le k$ $\iff$ $h^{r, s}_i = 0$ for $r\leq k$, $r+s \leq i-1$.  This is equivalent to: $\Gr_F^pW_{i-1}H^i(Y)=0$ for $p\le k$.   
 \end{proof}

 Recall that,  for any resolution $\pi:\hat Y\to Y$, 
\begin{equation}\label{intrepret-w1}
W_{i-1} H^i(Y)=\Ker \left(H^i(Y)\xrightarrow{\pi^*}H^i(\hat Y)\right)
\end{equation}
(e.g. \cite[Cor.\ 5.42]{PS}).  Using this, we obtain:
\begin{theorem}\label{cor-sym2}
 If $Y$ is a compact complex algebraic variety of dimension $n$ with either isolated or lci $k$-rational singularities, then 
 $$\uh^{p,q}=\uh^{q,p}.$$
for $0\le p\le k$ and $0\le q\le n$. 
\end{theorem}
\begin{proof}
In view of the discussion above, we define the discrepancy 
$$\delta_i^p:=\dim \Gr^p_F W_{i-1} H^i(Y)=\dim \Gr^p_F\Ker \left(H^i(Y)\xrightarrow{\pi^*}H^i(\hat Y)\right).$$
By Lemma~\ref{ineq-sym}, the equality $\uh^{p,i-p}=\uh^{i-p,p}$ holds for $p\le k$ $\iff$ $\delta_i^p=0$ for $p\le k$. 
The map $\psi^p$ occurring in the definition of higher rationality (Definition~\ref{defkrat} and  Lemma~\ref{lem-psi}) factors through the resolution $\pi:\hat Y\to Y$, and at the level of cohomology corresponds to the $\Gr^p_F$ piece (see also Corollary~\ref{cor-sym1}) of the natural map
$$\Psi^i:H^i(Y)\xrightarrow{\pi^*} H^i(\hat Y)\xrightarrow[PD]{\sim}H^{2n-i}(\hat Y)^\vee(-n)\xrightarrow{(\pi^*)^\vee} H^{2n-i}( Y)^\vee(-n),$$where all spaces are endowed with the natural Hodge structures. On the graded piece $\Gr_F^pH^i(Y) = H^{i-p}(Y; \uOp_Y)$, $\Gr_F^p\Psi^i$ is the map $\psi^p\colon H^{i-p}(Y; \uOp_Y) \to H^{n-i+p}(Y; \uOnp_Y)\spcheck$, which is an isomorphism if $Y$ is both $k$-rational and $k$-Du Bois, and in particular if $Y$ is $k$-rational and has either  isolated singularities or lci singularities.  
By the strictness of morphisms of Hodge filtrations, if $\Gr_F^p\Psi^i$ is an isomorphism then $\pi^*$ is injective on $\Gr_F^p$  and hence $\delta^p_i=0$.  Thus, $k$-rationality implies $\delta^p_i=0$ for $p\le k$ and all $i$, which in turn means $\uh^{p,i-p}=\uh^{i-p,i}$ in this range.   
\end{proof}

\begin{remark}
As noted by one of the referees, in the above proof as well as in the proof of Lemma~\ref{lem-psi}, the main point is to apply relative duality to a resolution of singularities morphism.  Thus the proof does not take into account the full information of a hyperresolution. 
\end{remark}

\section{Proof of   Theorem~\ref{mainThm} and Corollary~\ref{CYdef}}\label{section4}
We turn now to the global setting of a deformation of a compact analytic space or proper scheme $Y$, and to the question of the local freeness of $\Omega^p_{\cY/S}$.

\begin{theorem}\label{thm5} Let $f\colon \cY \to \Spec A$ be a proper morphism of complex spaces, where $A$ is an Artin local $\Cee$-algebra, with closed fiber $Y$.  Let $(\mathcal{F}^\bullet, d)$ be a bounded complex of coherent sheaves on $\cY$, flat over $A$, where $d\colon \mathcal{F}^i\to \mathcal{F}^{i+1}$ is $A$-linear, but not necessarily $\scrO_{\cY}$-linear. Finally suppose that the natural map $\mathbb{H}^i(\cY; \mathcal{F}^\bullet) = \Ar^if_*\mathcal{F}^\bullet \to \mathbb{H}^i(Y; \mathcal{F}^\bullet|Y)$ is surjective for all $i$. Then $\mathbb{H}^i(\cY; \mathcal{F}^\bullet)$ is a finite    $A$-module whose length  satisfies:
$$\ell(\mathbb{H}^i(\cY; \mathcal{F}^\bullet)) = \ell(A)\dim \mathbb{H}^i(Y; \mathcal{F}^\bullet|Y).$$
\end{theorem}
\begin{proof} This is a minor variation on very standard arguments. Note that $\mathbb{H}^i(\cY; \mathcal{F}^\bullet)$ is a finite $A$-module since it is the abutment of a spectral sequence with $E_1$ page $E_1^{p,q}= H^q(\cY; \mathcal{F}^p)$ and such that all of the differentials in the spectral sequence are $A$-module homomorphisms. Next we show that, for every finite $A$-module $M$, the natural map 
$$\Psi_M\colon \mathbb{H}^i(\cY; \mathcal{F}^\bullet)\otimes_AM \to \mathbb{H}^i(\cY; \mathcal{F}^\bullet\otimes _AM)$$ is surjective. The proof is via induction on $\ell(M)$, the case $\ell(M) =1$ being the hypothesis of the theorem. The inductive step follows from: given an exact sequence $0\to M' \to M\to M'' \to 0$ such that $\Psi_{M'}$ and $\Psi_{M''}$ are surjective, then $\Psi_M$ is surjective. This follows from the commutative diagram
$$\begin{CD} 
\mathbb{H}^i(\cY; \mathcal{F}^\bullet)\otimes_AM' @>>> \mathbb{H}^i(\cY; \mathcal{F}^\bullet)\otimes_AM @>>> \mathbb{H}^i(\cY; \mathcal{F}^\bullet)\otimes_AM'' @>>> 0\\
@VV{\Psi_{M'}}V @VV{\Psi_M}V @VV{\Psi_{M''}}V @.\\
\mathbb{H}^i(\cY; \mathcal{F}^\bullet\otimes _AM') @>>> \mathbb{H}^i(\cY; \mathcal{F}^\bullet\otimes _AM)@>>> \mathbb{H}^i(\cY; \mathcal{F}^\bullet\otimes _AM'') @.
\end{CD}$$
where the top row is exact since tensor product is right exact and the second is exact since $\mathcal{F}^\bullet$ is $A$-flat.

To prove the theorem, we argue by induction on $\ell(A)$. The result is clearly true if $\ell(A) = 1$. For the inductive step, write $A$ as a small extension $0 \to I \to A \to A/I \to 0$, so that $I\cong A/\mathfrak{m}$ and $\mathcal{F}^\bullet \otimes_A I  \cong \mathcal{F}^\bullet \otimes_A (A/\mathfrak{m})=\mathcal{F}^\bullet|Y$. By flatness, there is an exact sequence
$$\cdots \to \mathbb{H}^i(\cY; \mathcal{F}^\bullet \otimes_A I) \to \mathbb{H}^i(\cY; \mathcal{F}^\bullet) \to \mathbb{H}^i(\cY; \mathcal{F}^\bullet \otimes_AA/I) \to \cdots $$
Then the above implies that $\mathbb{H}^i(\cY; \mathcal{F}^\bullet) \to \mathbb{H}^i(\cY; \mathcal{F}^\bullet \otimes_AA/I)$ is surjective for all $i$. Hence the long exact sequence breaks up into short exact sequences and thus
\begin{align*}
\ell(\mathbb{H}^i(\cY; \mathcal{F}^\bullet)) &= \ell(\mathbb{H}^i(\cY; \mathcal{F}^\bullet\otimes_AI ))+ \ell(\mathbb{H}^i(\cY; \mathcal{F}^\bullet\otimes_AA/I))\\
&=  \dim \mathbb{H}^i(Y; \mathcal{F}^\bullet|Y) + \ell(A/I)\dim \mathbb{H}^i(Y; \mathcal{F}^\bullet|Y)\\
&= \ell(A)\dim \mathbb{H}^i(Y; \mathcal{F}^\bullet|Y).
\end{align*}
This concludes the inductive step.
\end{proof}

Now assume that  $f\colon \cY \to \Spec A$ is a proper morphism, where $\cY$ is a scheme of finite type over $\Spec \Cee$ and $A$ is an Artin local $\Cee$-algebra, with closed fiber $Y$. Consider the complex $\Omega^\bullet_{\cY/\Spec A}$ and the quotient complex $\Omega^\bullet_{\cY/\Spec A}/\sigma^{\geq k+1}\Omega^\bullet_{\cY/\Spec A}$. For the closed fiber $Y$, we also have the complex $\Omega^\bullet_Y/\sigma^{\geq k+1}\Omega^\bullet_Y$. 

\begin{lemma}\label{lemma2} With notation as above, suppose that $Y$ is $k$-Du Bois. Then:
\begin{enumerate}
\item[\rm(i)] For every $i$, the natural map 
$$\mathbb{H}^i(\cY; \Omega^\bullet_{\cY/\Spec A}/\sigma^{\geq k+1}\Omega^\bullet_{\cY/\Spec A}) \to \mathbb{H}^i(Y; \Omega^\bullet_Y/\sigma^{\geq k+1}\Omega^\bullet_Y)$$
is surjective.
\item[\rm(ii)] The spectral sequence with $E_1$ term
$$E_1^{p,q} = \begin{cases} H^q(Y; \Omega^p_Y) , &\text{for $p\leq k$;}\\
0, &\text{for $p> k$.}
\end{cases}$$
converging to $\mathbb{H}^{p+q}(Y; \Omega^\bullet_Y/\sigma^{\geq k+1}\Omega^\bullet_Y)$ degenerates at $E_1$. Hence, for every $i$,
$$\dim \mathbb{H}^i(Y; \Omega^\bullet_Y/\sigma^{\geq k+1}\Omega^\bullet_Y) = \sum_{\substack {p+q = i\\ p \leq k}}\dim H^q(Y; \Omega^p_Y).$$
\end{enumerate}
\end{lemma}
\begin{proof} By the $k$-Du Bois condition,  the natural map $\Omega^\bullet_Y/\sigma^{\geq k+1}\Omega^\bullet_Y \to \underline{\Omega}^\bullet_Y/F^{k+1}\underline{\Omega}^\bullet_Y$ is a quasi-isomorphism of filtered complexes.  Thus there are isomorphisms 
\begin{align*} \mathbb{H}^i(Y;\Omega^\bullet_Y/\sigma^{\geq k+1}\Omega^\bullet_Y) &\cong \mathbb{H}^i(Y; \underline{\Omega}^\bullet_Y/F^{k+1}\underline{\Omega}^\bullet_Y)  \\
H^q(Y;\Omega^p_Y) &\cong \mathbb{H}^q(Y; \underline{\Omega}^p_Y) \qquad \qquad (p\leq k).
\end{align*}
 By Corollary~\ref{DBdegen}(ii), the spectral sequence with $E_1$ page $\mathbb{H}^q(Y; \underline{\Omega}^p_Y)$ for $p\leq k$ and $0$ otherwise degenerates at $E_1$. Hence the same is true for the spectral sequence in (ii).

To prove (i), arguing as in \cite{DBJ}, there is a commutative diagram
$$\begin{CD}
H^i(\cY; \Cee) @>>> \mathbb{H}^i(\cY; \Omega^\bullet_{\cY/\Spec A}) @>>> \mathbb{H}^i(\cY; \Omega^\bullet_{\cY/\Spec A}/\sigma^{\geq k+1}\Omega^\bullet_{\cY/\Spec A}) \\
@| @VVV @VVV \\
H^i(Y; \Cee) @>>> \mathbb{H}^i(Y; \Omega^\bullet_Y) @>>> \mathbb{H}^i(Y; \Omega^\bullet_Y/\sigma^{\geq k+1}\Omega^\bullet_Y).
\end{CD}$$
 By Corollary~\ref{DBdegen}(i), $H^i(Y; \Cee) \to \mathbb{H}^i(Y; \Omega^\bullet_Y/\sigma^{\geq k+1}\Omega^\bullet_Y)\cong \mathbb{H}^i(Y; \underline{\Omega}^\bullet_Y/F^{k+1}\underline{\Omega}^\bullet_Y)$ is surjective. Hence 
$$\mathbb{H}^i(\cY; \Omega^\bullet_{\cY/\Spec A}/\sigma^{\geq k+1}\Omega^\bullet_{\cY/\Spec A}) \to \mathbb{H}^i(Y; \Omega^\bullet_Y/\sigma^{\geq k+1}\Omega^\bullet_Y)$$ is surjective as well.   
\end{proof}

 \begin{proof}[Proof of Theorem~\ref{mainThm}] By Corollary~\ref{cor1}, possibly after shrinking $S$, $\Omega^p_{\cY/S}$ is flat over $S$. By a standard reduction to the case of Artin local algebras, it is enough to show the following: Let $A$ be an Artin local $\Cee$-algebra and $f\colon \cY\to \Spec A$ a proper morphism whose closed fiber $Y$ is isomorphic to $Y_s$. Then $R^qf_*\Omega^p_{\cY/\Spec A}= H^q(\cY; \Omega^p_{\cY/\Spec A})$ is a free $A$-module, compatible with base change. In fact, by \cite[Theorem 12.11]{Hartshorne},   it is enough to show that, for every $q$ and quotient $A \to A/I=\overline{A}$, the natural map  $H^q(\cY; \Omega^p_{\cY/\Spec A}) \to H^q(\overline{\cY}; \Omega^p_{\overline{\cY}/\Spec \overline{A}})$ is surjective, where $\overline{\cY} = \cY\times_{\Spec A}\Spec \overline{A}$. 

By Theorem~\ref{thm5} and Lemma~\ref{lemma2}, for every $i$,
\begin{align*} 
\ell(\mathbb{H}^i(\cY; \Omega^\bullet_{\cY/\Spec A}/\sigma^{\geq k+1}\Omega^\bullet_{\cY/\Spec A})) &= \ell(A) \dim \mathbb{H}^i(Y; \Omega^\bullet_Y/\sigma^{\geq k+1}\Omega^\bullet_Y)\\
&= \ell(A)\sum_{\substack {p+q = i\\ p \leq k}}\dim H^q(Y; \Omega^p_Y).
\end{align*}
On the other hand, by analogy with Lemma~\ref{lemma2}(ii), there is the spectral sequence converging to $\mathbb{H}^i(\cY; \Omega^\bullet_{\cY/\Spec A}/\sigma^{\geq k+1}\Omega^\bullet_{\cY/\Spec A})$ with $E_1$ page 
$$E_1^{p,q} = \begin{cases} H^q(\cY; \Omega^p_{\cY/\Spec A}) , &\text{for $p\leq k$;}\\
0, &\text{for $p> k$.}
\end{cases}$$
Thus $\Dis\ell(\mathbb{H}^i(\cY; \Omega^\bullet_{\cY/\Spec A}/\sigma^{\geq k+1}\Omega^\bullet_{\cY/\Spec A})) \leq \sum_{\substack {p+q = i\\ p \leq k}}\ell(H^q(\cY; \Omega^p_{\cY/\Spec A}))$, with equality $\iff$ the spectral sequence degenerates at $E_1$. 
Moreover, a straightforward induction on $\ell(A)$ along the lines of the proof of Theorem~\ref{thm5} shows that $\ell(H^q(\cY; \Omega^p_{\cY/\Spec A})) \leq \ell(A) \dim H^q(Y; \Omega^p_Y)$, with equality holding for all $q$ $\iff$ for every $q$ and every quotient $A \to A/I=\overline{A}$, the natural map  $H^q(\cY; \Omega^p_{\cY/\Spec A}) \to H^q(\overline{\cY}; \Omega^p_{\overline{\cY}/\Spec \overline{A}})$ is surjective.  Combining, we have
\begin{align*} 
\ell(\mathbb{H}^i(\cY;& \Omega^\bullet_{\cY/\Spec A}/\sigma^{\geq k+1} \Omega^\bullet_{\cY/\Spec A}))   \leq   \sum_{\substack {p+q = i\\ p \leq k}}\ell(H^q(\cY; \Omega^p_{\cY/\Spec A})) \\
 &\leq \sum_{\substack {p+q = i\\ p \leq k}}\ell(A) \dim H^q(Y; \Omega^p_Y)   
 = \ell(A)\dim \mathbb{H}^i(Y; \Omega^\bullet_Y/\sigma^{\geq k+1}\Omega^\bullet_Y)  \\
 &= \ell(\mathbb{H}^i(\cY; \Omega^\bullet_{\cY/\Spec A}/\sigma^{\geq k+1}\Omega^\bullet_{\cY/\Spec A}).
\end{align*}
Thus all of the inequalities must have been equalities, and in particular $H^q(\cY; \Omega^p_{\cY/\Spec A}) \to H^q(\overline{\cY}; \Omega^p_{\overline{\cY}/\Spec \overline{A}})$ is surjective for every $q$ and for every quotient $A \to \overline{A}$. This completes the proof, by the remarks in the first paragraph of the proof. 
\end{proof}

\begin{remark}\label{rem-can} Instead of assuming that $Y$ is a scheme of finite type proper over $\Spec \Cee$, it is  enough to assume that $Y$ is a compact analytic space with isolated singularities and that there exists  a resolution  of singularities of $Y$ which satisfies  the $\partial\bar\partial$-lemma. Roughly speaking, the argument is as follows: The main point is to construct the analogue of the complex $\uOb_Y$ in this case, using Hironaka's resolution of singularities, and to note that all of the smooth varieties arising from a cubical hyperrresolution satisfy the $\partial\bar\partial$-lemma. Such varieties arise in two different ways: first, there is one such which is a resolution of singularities of $Y$. A standard argument (cf.\ \cite[proof of Theorem 2.29]{PS}) shows that in fact every resolution  of singularities of $Y$ satisfies the $\partial\bar\partial$-lemma. The remaining varieties in the construction can all be assumed to be smooth projective varieties coming from proper transforms of blowups of subvarieties of the exceptional divisors of the blowup of $Y$ along $Y_{\text{\rm{sing}}}$ or by iterating this procedure. In particular,  for every such variety, its cohomology carries a pure Hodge structure and the restriction morphisms between the various  cohomology groups are morphisms of Hodge structures.  From this, it follows that the Hodge spectral sequence $E_1^{p,q} =\mathbb{H}^q(Y; \uOp_Y) \implies H^{p+q}(Y;\Cee)$ degenerates at $E_1$. Then the argument of Theorem~\ref{mainThm} applies. 

It seems likely that there is a more general result, in the non-isolated case, assuming sufficiently strong conditions on  every subvariety of $Y_{\text{\rm{sing}}}$.    
\end{remark}

We turn now to a proof of Corollary~\ref{CYdef}. First, we define canonical Calabi-Yau varieties following \cite[Definition 6.1]{FL}:

\begin{definition}\label{defcanonCY} A \textsl{canonical Calabi-Yau variety} $Y$ is a scheme proper over $\Spec \Cee$ (resp.\ a compact analytic space) which is reduced,   Gorenstein, with canonical singularities, and  such that $\omega_Y \cong \scrO_Y$.
\end{definition}

\begin{proof}[Proof of Corollary~\ref{CYdef}] It suffices by \cite{Kawamata} to show that $Y$ has the $T^1$-lifting property. In fact, we show the following somewhat stronger statement: let $A$ be an  Artin local $\Cee$-algebra, $\pi\colon \mathcal{Y} \to \Spec A$ a deformation of $Y$ over $A$, and  $I$ an ideal of $A$ with $\overline{A} =A/I$ with $\overline{\cY} = \cY\times _{\Spec A}\Spec \overline{A}$, 
the natural map 
$$\Ext^1(\Omega^1_{\cY/\Spec A} , \scrO_{\cY}) \to \Ext^1(\Omega^1_{\overline{\cY}/\Spec \overline{A}} , \scrO_{\overline{\cY}})$$
is surjective.   Arguing as in \cite[Lemma 3]{Kawamata}, we claim that $\omega_{\cY/\Spec A} \cong \scrO_{\cY}$. To see this, note that, as $Y$ is $1$-Du Bois it is $0$-Du Bois, and hence by Theorem~\ref{Thm-duBois} $R^i\pi_*\scrO_{\cY}=H^i(\cY; \scrO_{\cY})$ is a free $A$-module for all $i$, compatible with base change. In particular, $R^n\pi_*\scrO_{\cY}$ is a free rank one $A$-module. An application of relative duality then shows that $(R^i\pi_*\scrO_{\cY})\spcheck \cong R^{n-i}\omega_{\cY/\Spec A}$, and hence $R^{n-i}\omega_{\cY/\Spec A}$ is a free $A$-module for all $i$. Thus $R^0\pi_*\omega_{\cY/\Spec A}$ is a free rank one $A$-module and the natural map  $R^0\pi_*\omega_{\cY/\Spec A} \otimes_A \scrO_{\cY} \to \omega_{\cY/\Spec A}$ is an isomorphism. Hence $\omega_{\cY/\Spec A} \cong \scrO_{\cY}$.

By a similar application of relative duality, and since $R^i\pi_*\Omega^1_{\cY/\Spec A} = H^i(\cY; \Omega^1_{\cY/\Spec A})$ is a free $A$-module for all $i$, there is an isomorphism of $A$-modules
\begin{align*} \Ext^1(\Omega^1_{\cY/\Spec A} , \scrO_{\cY}) &\cong \textit{Ext}_{\pi}^1(\Omega^1_{\cY/\Spec A} , \omega_{\cY/\Spec A}) \\
&\cong \Hom_A(H^{n-1}(\cY;\Omega^1_{\cY/\Spec A} ), A) .
\end{align*}
Then the $T^1$-lifting criterion follows from the fact that $H^{n-1}(\cY;\Omega^1_{\cY/\Spec A})$ is a free $A$-module and that the natural map $H^{n-1}(\cY;\Omega^1_{\cY/\Spec A}) \otimes _A\overline{A} \to H^{n-1}(\overline{\cY};\Omega^1_{\overline{\cY}/\Spec \overline{A}})$ is an isomorphism. 
\end{proof} 

\begin{remark} A similar argument using Remark~\ref{rem-can} shows that, if $Y$ is a compact analytic space with isolated lci $1$-Du Bois singularities such that $\omega_Y \cong \scrO_Y$ and there exists a resolution of singularities of $Y$ satisfying the $\partial\bar\partial$-lemma, then $\mathbf{Def} (Y)$ is unobstructed.
\end{remark} 

\section{Proof of Theorem~\ref{kratkDB}}\label{section5}

For $X$ an algebraic variety, let $\Sigma_k= \Sigma_k(X)$ be the set of points where $X$ is not $k$-Du Bois. Thus, for all $k$, $\Sigma_k \subseteq \Sigma_{k+1} \subseteq \Sigma$. It is easy to see that $\Sigma_k$ is a closed subvariety of $X$. In fact,  completing the morphism $\phi^p\colon \Omega^p_X \to \uOp_X $ to a distinguished triangle
$$ \Omega^p_X \to \uOp_X \to \mathcal{G}^p   \xrightarrow{+1},$$
by definition we have
$$\Sigma_k = \bigcup_{\substack{ i, p\\ 0\le p \le k}}\operatorname{Supp}\mathcal{H}^i\mathcal{G}^p.$$

In order to prove that $k$-rational singularities are $k$-Du Bois, we will need to consider situations more general than $k$-rational (as in the statement of Theorem~\ref{mainkratthm}). Recall that a \textsl{left inverse} to the map $\phi^p\colon \Omega^p_X \to \uOp_X $ is a map $h^p\colon \uOp_X \to \Omega^p_X$ such that $h^p\circ \phi^p =\Id$. Of course, left inverses need not exist in general.  More generally, we consider   a  set of $k+1$  left inverses $\{h^p\}_{p=0}^k$. If $X$ is $k$-Du Bois, then $\phi^p$ is an isomorphism for $p\le k$ and so $\{h^p\}_{p=0}^k =\{(\phi^p)^{-1}\}_{p=0}^k$ is a set of left inverses. In this case, we shall always use $\phi^p$ to identify $\Omega^p_X$ with  $\uOp_X $, and thus $h^p =\Id$ for $p \le k$.

If there exists a left inverse $H_k \colon\uOb_X /F^{k+1} \to \Omega^\bullet_X/\sigma^{\geq k+1} $ in the filtered derived category, then $\{\Gr^p H_k\}_{p=0}^k$ defines a set of $k+1$ left inverses. The following lemma shows that, conversely, given set of $k+1$  left inverses  $\{h^p\}_{p=0}^k$, we can modify them so as to arrange a left inverse $H_k \colon\uOb_X /F^{k+1} \to \Omega^\bullet_X/\sigma^{\geq k+1} $ in the filtered derived category.

\begin{lemma}\label{lemma5.2}    If the map $\Omega^\bullet_X/\sigma^{\geq k+1} \to \uOb_X /F^{k+1}$ has a set of $k+1$ left inverses $\{h^p\}_{p=0}^k$, then there exists a left inverse $H_k \colon\uOb_X /F^{k+1} \to \Omega^\bullet_X/\sigma^{\geq k+1} $ in the filtered derived category.
\end{lemma}
\begin{remark} We do not claim that the left inverse $H_k$ constructed in the proof satisfies $h^p = \Gr^p H_k$ for all $p \le k$.
\end{remark} 
\begin{proof}[Proof of Lemma~\ref{lemma5.2}]  We argue by induction on $k$, where the case $k=0$ is obvious. There is a diagram
$$\begin{CD}
\underline{\Omega}_X^k[-k]  @>>> \uOb_X/F^{k+1} @>>> \uOb_X/F^k  @>{+1}>> \\
@VV{h^k}V   @.  @VV{H_{k-1}}V  \\
\Omega^k_X[-k]   @>>> \Omega^\bullet_X/\sigma^{\geq k+1}  @>>> \Omega^\bullet_X/\sigma^{\geq k}  @>{+1}>>.
 \end{CD}$$ 
 We can thus complete the diagram to find  a morphism $H_k'\colon \uOb_X/F^{k+1} \to \Omega^\bullet_X/\sigma^{\geq k+1}$ in the filtered derived category which yields a morphism of distinguished triangles. The composition $G_k\colon \Omega^\bullet_X/\sigma^{\geq k+1} \to \uOb_X/F^{k+1} \xrightarrow{H_k'} \Omega^\bullet_X/\sigma^{\geq k+1}$ is an  isomorphism since it is an isomorphism on the  graded pieces. Set     $H_k = G_k^{-1}\circ H_k'\colon \uOb_X/F^{k+1} \to \Omega^\bullet_X/\sigma^{\geq k+1}$. Then   $H_k$   is a left inverse to the map $\Omega^\bullet_X/\sigma^{\geq k+1} \to \uOb_X/F^{k+1}$. 
\end{proof}

Next we define a special class of left inverses: 
\begin{definition} 
Let $\alpha \in H^0(X; \Omega^1_X)$. Then  $\alpha$ pulls back to some fixed hyperresolution and so defines a map $\alpha\wedge \colon \underline{\Omega}^{p-1}_X \to \uOp_X$. Two left inverses $h^p$ and $h^{p-1}$ are \textsl{compatible} if, for all $p\le k$ and all   $\alpha \in H^0(X; \Omega^1_X)$, 
$$h^p\circ (\alpha\wedge) = (\alpha\wedge)\circ h^{p-1}.$$ 
If $X$ is $k$-Du Bois, then, for all $p\le k$,  the left inverses $h^p$ and $h^{p-1}$ are isomorphisms. In fact, after identifying $\uOp_X$ and $\underline{\Omega}^{p-1}_X$ with $\Omega^p_X$ and $\Omega^{p-1}_X$ respectively, we can assume that $h^p =\Id$ for all $p \le k$. With this convention,  necessarily $h^p\circ (\alpha\wedge) = (\alpha\wedge)\circ h^{p-1}$, hence $h^p$ and $h^{p-1}$ are compatible.

Finally, the set of $k+1$  left inverses  $\{h^p\}_{p=0}^k$ is  \textsl{compatible} if, for all $p\le k$, $h^p$ and $h^{p-1}$ are compatible.
\end{definition}

\begin{lemma}\label{existsleftinv} If $X$ is $k$-rational, then there exists a compatible  set of $k+1$  left inverses $\{h^p\}_{p=0}^k$.
\end{lemma}
\begin{proof}
The following diagram commutes up to a sign (all vertical maps are $\alpha\wedge$):
$$\begin{CD}
\Omega^{p-1}_X @>>> \underline{\Omega}^{p-1}_X @>>> R\pi_*\Omega^{p-1}_{\hX} @>>> \mathbb{D}_X(\underline{\Omega}^{n-p+1}_X) \\
@VVV @VVV @VVV @VVV\\
\Omega^p_X @>>> \uOp_X @>>> R\pi_*\Omega^p_{\hX} @>>> \mathbb{D}_X(\underline{\Omega}^{n-p}_X)
\end{CD}$$
For $p\le k$, let $h^p$ be the left inverse to the map $\Omega^p_X \to \uOp_X$ defined by taking the inverse to the isomorphism $\gamma^p =\psi^p\circ \phi^p\colon \Omega^p_X \to \mathbb{D}_X(\underline{\Omega}^{n-p}_X)$ and composing  it with the map $\uOp_X\to \mathbb{D}_X(\underline{\Omega}^{n-p}_X)$. Since $\gamma^p\circ (\alpha\wedge) = (\alpha\wedge)\circ \gamma^{p-1}$, $(\gamma^p)^{-1}\circ (\alpha\wedge) = (\alpha\wedge)\circ (\gamma^{p-1})^{-1}$ and thus $h^p\circ (\alpha\wedge) = (\alpha\wedge)\circ h^{p-1}$. 
\end{proof}

We now deal with the case where $\dim \Sigma_k =0$, following the method of proof of Kov\'acs \cite[Theorem 2.3]{Kovacs99}, \cite[p.\ 186]{PS}:

\begin{proposition}\label{dim=0case} Suppose that $\dim \Sigma_k =0$ and that there is a left inverse $H_k\colon \uOb_X /F^{k+1} \to \Omega^\bullet_X/\sigma^{\geq k+1}$ to the map $\Omega^\bullet_X/\sigma^{\geq k+1} \to \uOb_X /F^{k+1}$. Then   $\Omega^\bullet_X/\sigma^{\geq k+1} \cong \uOb_X /F^{k+1}$.
\end{proposition}
\begin{proof}  Let $Y$ be some projective completion of $X$, so that $Y-X = D$. Let $Z = D \cup \Sigma_k \subseteq Y$. By hypothesis, $\Sigma_k$ is finite, hence $Z$ is closed. Let $U = X -\Sigma_k = Y - D - \Sigma_k$. Thus $U$ is $k$-Du Bois. Now consider the commutative diagram
$$\begin{CD}
\scriptstyle\mathbb{H}^{i-1}(U; \Omega^\bullet_U/\sigma^{\geq k+1}) @>>> \scriptstyle\mathbb{H}^i_Z(Y; \Omega^\bullet_Y/\sigma^{\geq k+1}) @>>> \scriptstyle\mathbb{H}^i(Y; \Omega^\bullet_Y/\sigma^{\geq k+1}) @>>> \scriptstyle\mathbb{H}^i(U; \Omega^\bullet_U/\sigma^{\geq k+1})\\
@VV{\cong}V @VVV @VVV @V{\cong}VV \\
\scriptstyle\mathbb{H}^{i-1}(U; \uOb_U/F^{k+1}) @>>>\scriptstyle \mathbb{H}^i_Z(Y; \uOb_Y/F^{k+1}) @>>> \scriptstyle\mathbb{H}^i(Y; \uOb_Y/F^{k+1}) @>>> \scriptstyle\mathbb{H}^i(U; \uOb_U/F^{k+1}).
\end{CD}$$
We claim that $\mathbb{H}^i(Y; \Omega^\bullet_Y/\sigma^{\geq k+1}) \to \mathbb{H}^i(Y; \uOb_Y/F^{k+1})$ is surjective: this is the usual argument that $H^i(Y;\Cee) \to \mathbb{H}^i(Y; \uOb_Y/F^{k+1})$ is surjective and factors through $H^i(Y;\Cee) \to \mathbb{H}^i(Y; \Omega^\bullet_Y/\sigma^{\geq k+1})$. Then $\mathbb{H}^i_Z(Y; \Omega^\bullet_Y/\sigma^{\geq k+1}) \to \mathbb{H}^i_Z(Y; \uOb_Y/F^{k+1})$ is surjective. We have compatible direct sum decompositions
$$\begin{CD}
\mathbb{H}^i_Z(Y; \Omega^\bullet_Y/\sigma^{\geq k+1}) @>{\cong}>> \mathbb{H}^i_D(Y; \Omega^\bullet_Y/\sigma^{\geq k+1}) \oplus \mathbb{H}^i_{\Sigma_k}(Y; \Omega^\bullet_Y/\sigma^{\geq k+1})\\
@VVV @VVV \\
\mathbb{H}^i_Z(Y; \uOb_Y/F^{k+1})  @>{\cong}>>  \mathbb{H}^i_D(Y; \uOb_Y/F^{k+1})\oplus  \mathbb{H}^i_{\Sigma_k}(Y; \uOb_Y/F^{k+1}).
\end{CD}$$
Hence $\mathbb{H}^i_{\Sigma_k}(Y; \Omega^\bullet_Y/\sigma^{\geq k+1}) \to \mathbb{H}^i_{\Sigma_k}(Y; \uOb_Y/F^{k+1})$ is surjective. By excision, $\mathbb{H}^i_{\Sigma_k}(Y; \Omega^\bullet_Y/\sigma^{\geq k+1}) \cong  \mathbb{H}^i_{\Sigma_k}(X; \Omega^\bullet_X/\sigma^{\geq k+1})$ and similarly for $\mathbb{H}^i_{\Sigma_k}(Y; \uOb_Y/F^{k+1})$. Thus,
$$ \mathbb{H}^i_{\Sigma_k}(X; \Omega^\bullet_X/\sigma^{\geq k+1}) \to \mathbb{H}^i_{\Sigma_k}(X; \uOb_X/F^{k+1})$$
is surjective. However, the existence of the left inverse $H_k$ gives a map $\mathbb{H}^i_{\Sigma_k}(X; \uOb_X/F^{k+1}) \to  \mathbb{H}^i_{\Sigma_k}(X; \Omega^\bullet_X/\sigma^{\geq k+1})$ such that the composed map $ \mathbb{H}^i_{\Sigma_k}(X; \Omega^\bullet_X/\sigma^{\geq k+1}) \to  \mathbb{H}^i_{\Sigma_k}(X; \Omega^\bullet_X/\sigma^{\geq k+1})$ is the identity.   Hence $$ \mathbb{H}^i_{\Sigma_k}(X; \Omega^\bullet_X/\sigma^{\geq k+1}) \to \mathbb{H}^i_{\Sigma_k}(X; \uOb_X/F^{k+1})$$
is also injective, and thus an isomorphism. Since  $\Omega^p_X|X-\Sigma_k \to \uOp_X|X-\Sigma_k$ is an isomorphism, it then follows from the ``Localization principle" of \cite[p.\ 186]{PS} that $\Omega^\bullet_X/\sigma^{\geq k+1} \cong \uOb_X /F^{k+1}$.
\end{proof} 

The following corollary then deals with the case $\dim \Sigma =0$ of Theorem~\ref{kratkDB}, in fact under the somewhat weaker hypothesis that $\dim \Sigma_k =0$.  

\begin{corollary} Suppose that $\dim \Sigma_k =0$ and that there exists a  left inverse $H_k \colon\uOb_X /F^{k+1} \to \Omega^\bullet_X/\sigma^{\geq k+1} $ to the map $\Omega^\bullet_X/\sigma^{\geq k+1} \to \uOb_X /F^{k+1}$. Then $X$ is $k$-Du Bois. In particular, if $X$ is $k$-rational and $\dim \Sigma_k =0$, then $X$ is $k$-Du Bois. 
\end{corollary}
\begin{proof}
The proof is by induction on $k$. There is a morphism of distinguished triangles
$$\begin{CD}
\Omega^k_X[-k]   @>>> \Omega^\bullet_X/\sigma^{\geq k+1} @>>> \Omega^\bullet_X/\sigma^{\geq k}   @>{+1}>> \\
 @VVV @VVV @VVV  \\
 \underline\Omega_X^k[-k]   @>>> \uOb_X/F^{k+1} @>>> \uOb_X/F^k @>{+1}>> .
 \end{CD}$$
 Here, the   center vertical arrow is an  isomorphism by Proposition~\ref{dim=0case} and the right vertical arrow is an isomorphism by induction. Thus, the left vertical arrow is an isomorphism as well.
 \end{proof}
 
 To handle the case where $\dim \Sigma_k > 0$, we consider the effect of passing to a general hyperplane section. Note that, if $H$ is an effective  Cartier divisor on $X$, since $\scrO_H$ is quasi-isomorphic to the complex $\scrO_X(-H) \to \scrO_X$,  we can define the operation $\otimes^{\mathbb{L}}\scrO_H$ on the (bounded) derived category,  and similarly for $\otimes^{\mathbb{L}}\scrO_H(-H)$. For a coherent sheaf $\mathcal{F}$ on $X$,  if $\mathit{Tor}_1^{\scrO_X}(\mathcal{F}, \scrO_H) =0$, then $\mathcal{F}\otimes^{\mathbb{L}}\scrO_H  = \mathcal{F}\otimes \scrO_H$. In particular, by Corollary~\ref{KDexactcor}, if $ \codim \Sigma \ge 2k-1$, then $\Omega^p_X\otimes^{\mathbb{L}}\scrO_H  = \Omega^p_X\otimes \scrO_H$ for all $p\le k+1$.
 
The following is due to Navarro Aznar \cite[V(2.2.1)]{GNPP}:

\begin{proposition}[Navarro Aznar]\label{NAV} If $X$ is an algebraic variety and $H$ is a general element of a base point free linear system on $X$, there is a distinguished triangle
$$\underline{\Omega}^{p-1}_H \otimes^{\mathbb{L}} \scrO_H(-H) \xrightarrow{df\wedge} \uOp_X \otimes^{\mathbb{L}}\scrO_H \to \uOp_H \xrightarrow{+1}. \qed$$
\end{proposition}

Here the term $\underline{\Omega}^{p-1}_H \otimes^{\mathbb{L}} \scrO_H(-H)$ refers to the tensor product as $\scrO_H$-modules, i.e.\ to 
$\underline{\Omega}^{p-1}_H \otimes^{\mathbb{L}}_{\scrO_H} \scrO_H(-H)$, and we could write this as $\underline{\Omega}^{p-1}_H \otimes \scrO_H(-H)$ or simply as $\underline{\Omega}^{p-1}_H$ in case $\scrO_H(-H) \cong \scrO_H$.

Recall that, in Proposition~\ref{KDexact},  we  defined the (not necessarily exact) sequence
\begin{equation*}
0 \to \Omega^{p-1}_H  \otimes  \scrO_H(-H) \to \Omega^p_X|H \to \Omega^p_H \to 0,\tag*{$(*)_p$}
\end{equation*}
Since the maps in the above distinguished triangle are clearly compatible with the augmentation maps $\phi^p$, we get:

\begin{corollary}\label{morphtriangles} If $(*)_p$ is   exact, then there is a morphism of distinguished triangles
$$\begin{CD}
\Omega^{p-1}_H \otimes \scrO_H(-H)  @>{df \wedge}>>  \Omega^p_X\otimes \scrO_H @>>> \Omega^p_H @>{+1}>> \\
@VVV @VVV @VVV @.\\
\underline{\Omega}^{p-1}_H  \otimes^{\mathbb{L}} \scrO_H(-H)  @>{df \wedge}>> \uOp_X \otimes^{\mathbb{L}}\scrO_H @>>>\uOp_H @>{+1}>>  
\end{CD}$$
 In particular, if $X$ is $(k-1)$-Du Bois and lci, then $(*)_p$ is   exact for $p \le k$ and hence there is a morphism of distinguished triangles as above for  $p \le k$.  
\end{corollary} 
\begin{proof} We only need to check the final statement. By Theorem~\ref{thmdim}, $\codim \Sigma  \geq 2k-1$. Thus, for a general $H$, by Proposition~\ref{KDexact}, the sequence $(*)_p$ is exact for all $p\le k$. 
\end{proof}

The key step is then the following:

\begin{theorem}\label{mainkratthm} Let $X$ be a reduced   local complete intersection and suppose that  there exists a compatible  set of $k+1$  left inverses $\{h^p\}_{p=0}^k$.  Then $X$ is $k$-Du Bois.
\end{theorem}
\begin{proof} We will assume that $\Sigma_k(X) = \Sigma_k \neq \emptyset$ and derive a contradiction. As this is a local question, we can assume that $X$ is affine. The proof is by induction both on $k$ and on $\dim \Sigma _k$. The cases $k=0$  and $\dim \Sigma _k =0$ have been proved by  Kov\'acs \cite{Kovacs99} or Proposition~\ref{dim=0case}, or follow by starting the induction at $k=-1$.  Assume inductively that the theorem has been proved for all $j < k$ and  for all varieties $H$ with $\dim \Sigma_k(H) < \dim \Sigma _k$. In particular, $X$ is $(k-1)$-Du Bois, and hence $\codim \Sigma  \geq 2k-1$. Choose a general $H$ as in Proposition~\ref{NAV}. Here the hypothesis that $H$ is general means in particular that (i) $\dim \Sigma_k \cap H < \dim \Sigma _k$ and (ii) if $\dim \Sigma _k >0$, then $\Sigma_k\cap H \neq \emptyset$. We claim that $\Sigma_k(H) \neq \emptyset$. For otherwise $H$ is $k$-Du Bois, hence $\Omega^p_H \cong \uOp_H$ for all $p\le k$. Then $\underline{\Omega}^k_X \otimes^{\mathbb{L}} \scrO_H\cong \Omega^k _X \otimes^{\mathbb{L}}\scrO_H$ as follows from the morphism of distinguished triangles in Corollary~\ref{morphtriangles}. On the other hand, we have the distinguished triangle
$$ \Omega^k _X\otimes^{\mathbb{L}} \scrO_H\to \underline{\Omega}^k_X  \otimes^{\mathbb{L}}\scrO_H\to \mathcal{G}^k\otimes^{\mathbb{L}}\scrO_H \xrightarrow{+1} .$$
Since $\scrO_H$ has the projective resolution $\scrO_X \xrightarrow{\times f} \scrO_X$, for all $p\le k$, and all $i$, there is an exact sequence
$$ 0 \to Tor_1^{\scrO_X}(\mathcal{H}^i\mathcal{G}^p, \scrO_H) \to \mathcal{H}^i(\mathcal{G}^p\otimes^{\mathbb{L}}\scrO_H) \to (\mathcal{H}^i\mathcal{G}^p )\otimes \scrO_H \to 0.$$
Thus $\displaystyle\bigcup_{\substack{ i, p\\ 0\le p \le k}}\operatorname{Supp} \mathcal{H}^i(\mathcal{G}^p\otimes^{\mathbb{L}}\scrO_H) =\Sigma_k\cap H \neq \emptyset$, and hence $\underline{\Omega}^k_X \otimes^{\mathbb{L}} \scrO_H\to \Omega^k _X \otimes^{\mathbb{L}}\scrO_H$ is not an isomorphism.

For all $p\le k$,  we claim that  we can find  a compatible  set of $p+1$  left inverses  $\{\bar{h}^i\}_{i=0}^p $. Again, we argue by induction on $p$ and the case $p=0$ is clear. For the inductive step, assume that there exists a compatible set of $p$  left inverses $\{\bar{h}^i\}_{i=0}^{p-1}$  for $H$. By the inductive hypothesis, $H$ and $X$ are $(p-1)$-Du Bois, and thus we can take $\bar{h}^i = \Id$ for all $i \le p-1$. Also, as noted above, by Corollary~\ref{KDexactcor}, $\Omega^p_X\otimes^{\mathbb{L}} \scrO_H= \Omega^p_X\otimes \scrO_H$.  In what follows, we fix the function $f$ and will omit the factor $\scrO_H(-H)$ as it is trivialized. 

By the assumption of compatibility, there is a commutative diagram
$$\begin{CD}
\Omega^{p-1}_X\otimes \scrO_H  @>{df \wedge}>> \uOp_X \otimes^{\mathbb{L}}\scrO_H\\
@VV{=}V @VV{h^p\otimes \Id}V \\
\Omega^{p-1}_X\otimes \scrO_H  @>{df \wedge}>> \Omega^p_X\otimes \scrO_H.
\end{CD}$$
There is an induced diagram
$$\begin{CD}
\Omega^{p-1}_X\otimes \scrO_H  @>>> \Omega^{p-1}_H @>{df \wedge}>> \uOp_X \otimes^{\mathbb{L}}\scrO_H\\
@VV{=}V @VV{=}V @VV{h^p\otimes \Id}V \\
\Omega^{p-1}_X\otimes \scrO_H  @>>> \Omega^{p-1}_H @>{df \wedge}>> \Omega^p_X\otimes \scrO_H.
\end{CD}$$
Here, the left hand square is commutative and the outer rectangle is commutative. Since all terms except for $\uOp_X \otimes^{\mathbb{L}}\scrO_H$ are sheaves and the morphism $\Omega^{p-1}_X\otimes \scrO_H  \to \Omega^{p-1}_H$ is surjective, a straightforward argument shows that the right hand square is commutative as well. 

There is a diagram of distinguished triangles
$$\begin{CD}
\underline{\Omega}^{p-1}_H\cong \Omega^{p-1}_H   @>{df \wedge}>> \uOp_X \otimes^{\mathbb{L}}\scrO_H @>>>\uOp_H @>{+1}>> \\
@VV{=}V @VV{h^p\otimes \Id}V @. @.\\
\Omega^{p-1}_H @>{df \wedge}>>  \Omega^p_X\otimes \scrO_H @>>> \Omega^p_H @>{+1}>> 
\end{CD}$$
 Thus, we can complete the diagram by finding $\bar{h}^p$ which makes the diagram commute, and it is automatically a left inverse to the map $\Omega^p_H \to \uOp_H$ since $\Omega^p_X\otimes \scrO_H \to \Omega^p_H$ is surjective. We claim that any such $\bar{h}^p$ is compatible with $\bar{h}^{p-1}=\Id$: Since $\underline{\Omega}^{p-1}_H \cong \Omega^{p-1}_H$, there is a commutative diagram with vertical arrows given by $\alpha\wedge$:
$$\begin{CD}
\underline{\Omega}^{p-1}_H \cong \Omega^{p-1}_H @>{=}>> \Omega^{p-1}_H \\
@VVV @VVV \\
\uOp_H @>{\bar{h}^p}>> \Omega^p_H
\end{CD}$$
As $\bar{h}^p$ is a left inverse, for all $\varphi \in \underline{\Omega}^{p-1}_H \cong \Omega^{p-1}_H $, $\bar{h}^p(\alpha\wedge \varphi) = \alpha\wedge \varphi = \alpha\wedge \bar{h}^{p-1}(\varphi)$. It follows that $\bar{h}^p$ is compatible with $\bar{h}^{p-1}$.   This completes the inductive step for $p$. 

But then, since $\dim \Sigma_k(H) < \dim \Sigma _k$, by the inductive hypothesis $H$ is $k$-Du Bois and hence $\Sigma_k(H) =\emptyset$. This  contradicts  the statement that $\Sigma_k(H) \neq\emptyset$.
\end{proof} 

The following is then the lci case of Theorem~\ref{kratkDB}:

\begin{corollary}\label{kratkDBlci} If $X$ is a $k$-rational algebraic variety with lci singularities, then $X$ is $k$-Du Bois.
\end{corollary}
\begin{proof} This follows from Lemma~\ref{existsleftinv} and Theorem~\ref{mainkratthm}. 
\end{proof}

In fact, the proof also shows the following result of Musta\c{t}\u{a}-Popa \cite[Theorem 9.17]{MP-loc}:

\begin{corollary} If $X$ is an algebraic variety with $k$-Du Bois singularities, then a general complete intersection $H_1\cap \cdots \cap H_a$ of $X$  is $k$-Du Bois. \qed
\end{corollary}

As an application of Corollary~\ref{kratkDBlci}, we have:

 \begin{proposition}\label{0.6}   Let $f:\cY\to S$ be a flat proper family of complex algebraic varieties of relative dimension $n$ over an  irreducible base $S$. For $s\in S$, suppose that the fiber $Y_s$ has $k$-rational  lci singularities.  Then,  for every fiber $t$ such that $Y_t$ is smooth, $\dim \Gr_F^p H^{p+q}(Y_t) = \dim \Gr_F^{n-p} H^{2n-p-q}(Y_s)$ for every $q$ and for $0\le p \le k$.  Equivalently, for such $p$ and $q$,  $$h^{p,q}(Y_t) = h^{n-p, n-q}(Y_t) = \uh^{p,q}(Y_s) = \uh^{n-p, n-q}(Y_s).$$
\end{proposition}
\begin{proof}  By Corollary~\ref{kratkDBlci}, $Y_s$ is  $k$-Du Bois.  By Theorem~\ref{mainThm}, for $p\le k$, $R^qf_*\Omega^p_{\cY/S}$ is locally free in a neighborhood of $s$ and compatible with base change.  Thus,
$$\dim \Gr_F^p H^{p+q}(Y_t) = H^q(Y_t; \Omega^p_{Y_t}) = \dim H^q(Y_s; \Omega^p_{Y_s}).$$
Then $\dim \Gr_F^p H^{p+q}(Y_t)  = \dim \mathbb{H}^q(Y_s; \bD_{Y_s}(\underline\Omega^{n-p}_{Y_s} ))$. By Grothendieck duality,
$\mathbb{H}^q(Y_s; \bD_{Y_s}(\underline\Omega^{n-p}_{Y_s} ))$ is dual to $\mathbb{H}^{n-q}(Y_s; \underline\Omega^{n-p}_{Y_s})$. 
Computing dimensions gives the result.
\end{proof}

 In fact, combining the above with the Hodge symmetries given by Theorem~\ref{cor-sym2}, we obtain
Corollary~\ref{Corkrat} announced in the introduction.  This is modeled on \cite[Thm. 1]{KLS} (case $k=0$). Note however that loc.\ cit.\ does not assume lci singularities, and works in the analytic category. 

\bigskip

\bigskip

\bigskip

\appendix
\section{Proof of  Conjecture \ref{conjDBrat} for hypersurfaces}\label{AppendixA}
\par\bigskip
\centerline{Morihiko Saito}
\smallskip
\centerline{\smaller\smaller RIMS Kyoto University, Kyoto 606-8502 Japan}
\par\bigskip\noindent
We prove that two definitions of higher $k$-rational singularities for hypersurfaces coincide, see Theorem~\ref{Thm-A} below. (The case $k\,{=}\,0$ was treated in \cite{Saito-b}.) This implies a proof of Conjecture~ \ref{conjDBrat} for hypersurfaces using the converse of a theorem of Musta\c t\u a, Olano, Popa, and Witaszek \cite[Thm.\,1.1]{MOPW} (see \cite[Thm.\,1]{JKSY-duBois}).
\par\medskip\noindent
\begin{theorem}\label{Thm-A}
Assume $X$ is a reduced hypersurface of a smooth complex algebraic variety $Y$. Then for $k\in{\mathbb Z}_{>0}$, we have $\widetilde{\alpha}_X>k{+}1$ if and only if $X$ has only $k$-rational singularities.
\end{theorem}
\begin{proof} Assume $\widetilde{\alpha}_X>k{+}1$. We may assume that $X\subset Y$ is defined by a function $f$ shrinking $X,Y$ if necessary. Since $\widetilde{\alpha}_X>k{+}1$, we have by \cite[Thm.\,2]{JKSY-duBois} the canonical isomorphism
$$\underline{\Omega}_X^{d_X-k}=\sigma_{\,{\geqslant}\, 0}\bigl(\Omega_Y^{\scriptscriptstyle\bullet}[d_Y{-}k]|_X,{\rm d} f\wedge\bigr),
\leqno{\rm(A1)}$$
where $d_Y:=\dim Y$. Applying the functor ${\mathbb D}$, we get the isomorphisms
$$\aligned&{\mathbb D}(\underline{\Omega}_X^{d_X-k})={\mathbb D}\bigl(\sigma_{\,{\geqslant}\, 0}(\Omega_Y^{\scriptscriptstyle\bullet}[d_Y{-}k]|_X,{\rm d} f\wedge)\bigr)\\ &=C\bigl(f:\sigma_{\,{\leqslant}\, 0}(\Omega_Y^{\scriptscriptstyle\bullet}[k],{\rm d} f\wedge)\to\sigma_{\,{\leqslant}\, 0}(\Omega_Y^{\scriptscriptstyle\bullet}[k],{\rm d} f\wedge)\bigr)[d_Y{-}1],\endaligned
\leqno{\rm(A2)}$$
since ${\mathbb D}(L)=L^{\vee}{\otimes}_{{\mathcal O}_Y}\omega_Y[d_Y]$ for a locally free ${\mathcal O}_Y$-module $L$ in general.
\par\smallskip
It is well known that
$${\mathcal H}^j(\Omega_Y^{\scriptscriptstyle\bullet},{\rm d} f\wedge)=0\quad\hbox{if}\quad j<{\rm codim}_Y{\rm Sing}\,X,
\leqno{\rm(A3)}$$
$$\widetilde{\alpha}_X<\tfrac{1}{2}\,{\rm codim}_Y{\rm Sing}\,X,
\leqno{\rm(A4)}$$
see for instance \cite[Prop.\,1--2]{JKSY-duBois}. These imply the quasi-isomorphism
$$\sigma_{\,{\leqslant}\, 0}(\Omega_Y^{\scriptscriptstyle\bullet}[k],{\rm d} f\wedge)\,\,\rlap{\hskip1.5mm\raise1.4mm\hbox{$\sim$}}\hbox{$\longrightarrow$}\,\,\Omega_Y^k/{\rm d} f{\wedge}\Omega_Y^{k-1},
\leqno{\rm(A5)}$$
together with the injection
$$\Omega_Y^k/{\rm d} f{\wedge}\Omega_Y^{k-1}\hookrightarrow\Omega_Y^{k+1},
\leqno{\rm(A6)}$$
which gives $f$-torsion-freeness of $\Omega_Y^k/{\rm d} f{\wedge}\Omega_Y^{k-1}$. We thus get the canonical isomorphism
$${\mathbb D}(\underline{\Omega}_X^{d_X-k})[-d_X]=\Omega_X^k.
\leqno{\rm(A7)}$$
Here $k$ can be replaced by any $j\in[0,k{-}1]$. So $X$ has only $k$-rational singularities.
\par\smallskip
Assume now $X$ has only $k$-rational singularities. This means that the composition
$$\Omega_X^k\to\underline{\Omega}_X^k\to{\mathbb D}(\underline{\Omega}_X^{d_X-k})[-d_X]
\leqno{\rm(A8)}$$
is an isomorphism (hence $\Omega_X^k$ is a direct factor of $\underline{\Omega}_X^k$) and the same holds with $k$ replaced by any $j\in[0,k{-}1]$. We will consider the morphisms obtained by applying the functor ${\mathbb D}$ to these morphisms. We argue by induction on $k$. Note that
$$\widetilde{\alpha}_X>k,
\leqno{\rm(A9)}$$
since $X$ has only $(k{-}1)$-rational singularities by definition. This implies that
$$k{+}1<{\rm codim}_Y{\rm Sing}\,X,
\leqno{\rm(A10)}$$
using (A4), since ${\rm codim}_Y{\rm Sing}\,X\,{\geqslant}\, 2$. By the same argument as above, we then get that
$${\mathbb D}(\Omega_X^k)[-d_X]=\sigma_{\,{\geqslant}\, 0}(\Omega_Y^{\scriptscriptstyle\bullet}[d_Y{-}k]|_X,{\rm d} f\wedge).
\leqno{\rm(A11)}$$
\par\smallskip
On the other hand we have by \cite[Thm.~4.2]{Saito2000}
$$\underline{\Omega}_X^k={\rm Gr}_F^k{\rm DR}\bigl({\mathbb Q}_{h,X}[d_X]\bigr)[k{-}d_X],
\leqno{\rm(A12)}$$
hence
$${\mathbb D}(\underline{\Omega}_X^k)[-d_X]={\rm Gr}_F^{d_X-k}{\rm DR}\,{\mathbb D}\bigl({\mathbb Q}_{h,X}(d_X)[d_X]\bigr)[d_X{-}k].
\leqno{\rm(A13)}$$
By the theory of Hodge ideals (see \cite{MP-Mem}, \cite{SaitoV}, \cite{JKSY}, \cite{JKSY-duBois}) and using (A9), we can get the isomorphism
$${\mathbb D}(\underline{\Omega}_X^k)[-d_X]\cong K^{(k),\scriptscriptstyle\bullet}\subset\sigma_{\,{\geqslant}\, 0}(\Omega_Y^{\scriptscriptstyle\bullet}[d_Y{-}k]|_X,{\rm d} f\wedge),
\leqno{\rm(A14)}$$
where $K^{(k),j}:=\Omega_Y^{j+d_Y-k}|_X$ if $j\ne k$, and $K^{(k),k}:=I_k(X)\Omega_Y^{d_Y}/f\Omega_Y^{d_Y}$ with $I_k(X)$ the Hodge ideal.
\par\smallskip
If $\widetilde{\alpha}_X\,{\in}\,(k,k+1)$ so that $I_k(X)\ne{\mathcal O}_Y$ (see for instance \cite[Cor.\,1]{SaitoV}), then by (A11), (A14) the canonical morphism
$${\mathcal H}^{k-d_X}{\mathbb D}(\underline{\Omega}_X^k)\to{\mathcal H}^{k-d_X}{\mathbb D}(\Omega_X^k)
\leqno{\rm(A15)}$$
is never surjective. Indeed, since $\Omega_X^k$ is a direct factor of $\underline{\Omega}_X^k$, the mapping cone of a morphism $\phi\,{:}\,\Omega_X^k\,{\to}\,\underline{\Omega}_X^k$ is independent of $\phi$ as long as it induces an isomorphism on $X\,{\setminus}\,{\rm Sing}\,X$. Note that ${\mathcal H}^0_{{\rm Sing}\,X}\Omega_X^k\,{=}\, 0$, since the proof of Prop.\,2.2 in \cite{JKSY-duBois} holds also for $q\,{=}\, p{+}1$ (where the last inequality in the proof of Prop.\,2.2 becomes $q{+}1\,{=}\, p{+}2<{\rm codim}_Y{\rm Sing}\,X$). Hence the dual of the composition (A8) cannot be an isomorphism.
\par\smallskip
Assume now $\widetilde{\alpha}_X\,{=}\,k{+}1$. (Here $\Omega_X^k\,{=}\,\underline{\Omega}_X^k$, see \cite{MOPW}, \cite{JKSY-duBois}.) The canonical isomorphism (A1) and the second morphism of (A8) are induced by the canonical morphism of mixed Hodge modules
$${\mathbb Q}_{h,X}[d_X]\to{\mathbb D}\bigl({\mathbb Q}_{h,X}(d_X)[d_X]\bigr),
\leqno{\rm(A16)}$$
see \cite[3.1]{JKSY-duBois}. Note that this coincides with the composition of $${\mathbb Q}_{h,X}[d_X]\,{\to}\,\rho_*{\mathbb Q}_{h,\widetilde{X}}[d_X]$$ with its dual, where $\rho\,{:}\,\widetilde{X}\,{\to}\, X$ is a desingularization. Let $(M',F),(M'',F)$ be the underlying filtered ${\mathcal D}_Y$-modules of its kernel and cokernel respectively. Then the condition $\widetilde{\alpha}_X\,{=}\, k{+}1$ implies that
$$\min\{p\in{\mathbb Z}\mid F_pM'\ne 0\}>d_X{-}k,\quad\min\{p\in{\mathbb Z}\mid F_pM''\ne 0\}=d_X{-}k,
\leqno{\rm(A17)}$$
using \cite[(1.3.2--4)]{SaitoV} and \cite[3.1]{JKSY-duBois} together with the $N$-primitive decomposition, see for instance \cite[(2.2.4)]{KLS}. We then see that the morphism (A16) {\it cannot} induce an isomorphism
$${\rm Gr}_F^{d_X-k}{\rm DR}\bigl({\mathbb Q}_{h,X}[d_X]\bigr)\,\,\rlap{\hskip1.5mm\raise1.4mm\hbox{$\sim$}}\hbox{$\longrightarrow$}\,\,{\rm Gr}_F^{d_X-k}{\rm DR}\bigl({\mathbb D}\bigl({\mathbb Q}_{h,X}(d_X)[d_X]\bigr)\bigr).
\leqno{\rm(A18)}$$
This means that the composition (A8) cannot be an isomorphism in view of (A12--A13). This is a contradiction. We thus get $\widetilde{\alpha}_X>k{+}1$. This finishes the proof of Theorem \ref{Thm-A}.
\end{proof}

Combining Theorem \ref{Thm-A} with the converse of a theorem of Musta\c t\u a, Olano, Popa, and Witaszek \cite[Thm.\,1.1]{MOPW} (see \cite[Thm.\,1]{JKSY-duBois}), we can get a positive answer to Conjecture~\ref{conjDBrat} for hypersurfaces as follows.
\begin{corollary}\label{Cor-A} Assume $X$ is a reduced hypersurface of a smooth complex algebraic variety $Y$, and has only $k$-Du Bois singularities $(k\,{\geqslant}\, 1)$. Then $X$ has only $(k{-}1)$-rational singularities.
\end{corollary}

\begin{remark}In \cite{JKSY-duBois} an assertion slightly stronger than the converse of \cite[Thm.\,1.1]{MOPW} is proved (since it is not assumed that the isomorphism is induced by the canonical morphism). The assertion can be proved rather easily using the extension of \cite[Prop.\,2.2]{JKSY-duBois} to the case $q\,{=}\, p{+}1$ as is explained after (A15) and assuming that the restriction of the isomorphism to the smooth points of $D$ is the identity. Indeed, the argument in \cite[2.3]{JKSY-duBois} is so complicated, since even this natural condition is not assumed.
\end{remark}

\bibliography{duBois}
\end{document}